\documentclass[12pt,a4paper,reqno]{amsart}
\usepackage[english]{babel}
\usepackage[utf8]{inputenc}
\usepackage{graphicx,tikz}
\usepackage{amssymb, amsmath}
\usepackage{geometry,enumerate,stackrel,mathtools}
\usepackage{enumitem,centernot,color}

\usepackage{multirow}

\usepackage{makecell} 

\usepackage{comment}

\usepackage{tabularx}
\usepackage{booktabs}
\usepackage{array}

\usepackage[colorlinks=true,linkcolor=blue,urlcolor=blue,citecolor=blue]{hyperref}

\newcolumntype{C}{>{\centering\arraybackslash}X}

\newtheorem{theorem}{Theorem}[section]

\newtheorem{lemma}[theorem]{Lemma}

\theoremstyle{definition}

\begin{document}

\title[On sharp constants for equivalent norms on Ces\`aro function spaces]{On sharp constants for infinitely many equivalent norms on Ces\`aro function spaces}

\author[A. Ben Said]{Achraf Ben Said$^{*}$}

\address{Department of Analysis and Applied Mathematics, Complutense University  of Madrid, 28040 Madrid, Spain.}
\email{achbensa@ucm.es}

\author[M. Monsalve-L\'opez]{Miguel Monsalve-L\'opez$^{**}$}

\address{Department of Analysis and Applied Mathematics, Complutense University  of Madrid, 28040 Madrid, Spain.}
\email{migmonsa@ucm.es}



\begin{abstract}
Recent years have seen significant interest in finding equivalent norms on Ces\`aro function spaces together with their optimal equivalence constants. In this work, we continue this line of research by deriving sharp two-sided inequalities relating the $L^p$-norms of $\mathcal{C}_m f$ and $\mathcal{C}_n f$, where $\mathcal{C}_k$ is the generalized Ces\`aro operator of order $k \in \mathbb{N}$ acting on non-negative measurable functions $f$ on $(0,\infty)$. A parallel study is also conducted for the corresponding adjoint operators.
\end{abstract}

\maketitle

\section{Introduction}
Let $\mathcal{M}(\mathbb{R}^+)$ denote the space of measurable functions on $\mathbb{R}^+ = (0,\infty)$, and let $\mathcal{M}^+(\mathbb{R}^+)$ be the subclass of non-negative functions in $\mathcal{M}(\mathbb{R}^+)$. For $m>0$, the \emph{generalized Ces\`aro operator of order $m$}, denoted by $\mathcal{C}_m$, and its dual operator $\mathcal{C}_m^*$ are defined for $f \in \mathcal{M}(\mathbb{R}^+)$ by
\[
\mathcal{C}_m f(x) = \frac{m}{x^m} \int_0^x f(t)(x-t)^{m-1} \, dt, \quad x > 0,
\]
and
\[
\mathcal{C}_m^* f(x) = m \int_x^\infty \frac{f(t)}{t^m} (t-x)^{m-1} \, dt, \quad x > 0,
\]
respectively, provided these integrals exist. By convention, for $m = 0$, both $\mathcal{C}_0$ and $\mathcal{C}_0^*$ are defined as the identity operator $I$, so that
\[
\mathcal{C}_0 f = \mathcal{C}_0^*f  = f.
\]

\smallskip

Recall that for $m > 0$, the classical \emph{Riemann--Liouville fractional integral operator} $J_a^m$ with origin at $a \in \mathbb{R}$ is defined as
\begin{equation*}
    (J_a^m f)(x) = \frac{1}{\Gamma(m)} \int_a^x f(t)(x-t)^{m-1} \, dt, \quad x > a
\end{equation*}
(see, e.g., \cite[Definition~2.1, p.~13]{Kai}), where $\Gamma \colon (0,\infty) \to \mathbb{R}$ denotes the classical \emph{Euler Gamma function}. Setting $a = 0$ yields a direct connection to $\mathcal{C}_m$, namely
\begin{equation*}
    \mathcal{C}_m f(x) = \frac{\Gamma(m+1)}{x^m} (J_0^m f)(x).
\end{equation*}
Analogously, the adjoint $\mathcal{C}_m^*$ connects to the \emph{Weyl fractional integral operator $W^m$ of order $m > 0$} \cite[p.~239]{Miller}, defined on $x > 0$ by
\begin{equation*}
    (W^m f)(x) = \frac{1}{\Gamma(m)} \int_x^\infty f(t)(t-x)^{m-1} \, dt.
\end{equation*}
Indeed, the adjoint formulation takes the form
\begin{equation*}
    \mathcal{C}^*_m f = \Gamma(m+1) W^m \left( \frac{f(\cdot)}{(\cdot)^m} \right).
\end{equation*}

 The generalized Ces\`aro operators $\mathcal{C}_m$ and their adjoints play an important role in operator theory and harmonic analysis and have been widely studied. For instance, a detailed study of the boundedness and a spectral analysis  of these operators on Sobolev--Lebesgue spaces was carried out in~\cite{Liz}. For further studies regarding these operators, the reader is referred to~\cite{X,Boyd, Brown, Mori} and the references therein.

\medskip

The boundedness of the Ces\`aro operator $\mathcal{C}_1$ from $L^p(\mathbb{R}^+)$ into itself is a straightforward consequence of the well-known \emph{Hardy inequality}, see a nice survey about this inequality in~\cite{Kuf}. The following generalized Hardy inequality, see \cite[Theorem~329, p.~245]{HLP}, allows to show the boundedness of Ces\`aro operators $\mathcal{C}_m$ from $L^p(\mathbb{R}^+)$ into itself:

Let $m>0$. For every $f \in L^p(\mathbb{R}^+)$ with $1 < p \le \infty$, the operator $\mathcal{C}_m$ satisfies the bound
\begin{equation}
\label{BoundednessC}
\left( \int_0^\infty \left| \frac{m}{x^m} \int_0^x f(t)(x-t)^{m-1} \, dt \right|^p \, dx \right)^{1/p} 
\le \frac{\Gamma(m+1)\Gamma(1/p')}{\Gamma(m+1/p')} \left( \int_0^\infty |f(x)|^p \, dx \right)^{1/p},
\end{equation}
where $p' = p/(p-1)$ is the conjugate exponent of $p$. Moreover, the constant in \eqref{BoundednessC} is the best possible. 
A closer (and dual) inequality is the following, 
\begin{equation}
\label{BoundednessC*}
\left( \int_0^\infty \left| m \int_x^\infty (t-x)^{m-1} f(t) \, \frac{dt}{t^m} \right|^p \, dx \right)^{1/p} 
\le \frac{\Gamma(m+1)\Gamma(1/p)}{\Gamma(m+1/p)} \left( \int_0^\infty |f(x)|^p \, dx \right)^{1/p},
\end{equation}
which holds for all $f \in L^p(\mathbb{R}^+)$ with $1 \le p < \infty$. Also the constant in inequality \eqref{BoundednessC*} is optimal; see \cite[Theorem~329, p.~245]{HLP}.

\medskip

The classical \emph{Ces\`aro function spaces $\Lambda_1^p(\mathbb{R}^+)$} and the classical \emph{Copson function spaces $\Lambda_1^{*p}(\mathbb{R}^+)$} are defined, respectively, by
\[
\Lambda_1^p(\mathbb{R}^+)
= \bigl\{ f \in \mathcal{M}(\mathbb{R}^+): \|f\|_{\Lambda_1^{p}(\mathbb{R}^+)} = \|\mathcal{C}_1|f|\|_p < \infty \bigr\},
\qquad 1 < p \le \infty,
\]
and
\[
\Lambda_1^{*p}(\mathbb{R}^+)
= \bigl\{ f \in \mathcal{M}(\mathbb{R}^+): \|f\|_{\Lambda_1^{*p}(\mathbb{R}^+)} = \|\mathcal{C}_1^*|f|\|_p < \infty \bigr\},
\qquad 1 \le p < \infty.
\]

\medskip

These spaces are Banach spaces that have been extensively studied; see, for example, \cite{As, B, Jag, Mal, Si2} and the references therein.

\medskip

Note that, in particular, inequality \eqref{BoundednessC} implies that 
\[
L^p(\mathbb{R}^+) \subset \Lambda_1^p(\mathbb{R}^+)
\]
for all $1 < p \le \infty$. However, the reverse inclusion does not hold. Indeed, because the classical Hardy operator $\mathcal{C}_1$ is not invertible on $L^p(\mathbb{R}^+)$ (see \cite{Milman}), there exists no constant $c(p) > 0$ depending solely on $p$ such that the reverse inequality
\begin{equation}
\label{IntroEq3}
\|\mathcal{C}_1 f\|_{p} \ge c(p) \, \|f\|_{p}
\end{equation}
holds in general, even for non-negative functions $f \in L^p(\mathbb{R}^+)$. Interestingly, replacing $\mathcal{C}_1$ with $\mathcal{C}_1 - I$ yields an operator that is both continuous and bounded from below on $L^p(\mathbb{R}^+)$ for every $1 < p < \infty$ (see~\cite{ABS2}). 

\medskip

Nevertheless, if we replace the $L^p(\mathbb{R}^+)$ norm with the $\Lambda_1^p(\mathbb{R}^+)$ norm in \eqref{IntroEq3}, we have the equivalence
\begin{equation}
\label{Equivalence1}
\frac{\Gamma(p+1)^{1/p}}{p-1} \|f\|_{\Lambda_1^p(\mathbb{R}^+)} \le \|\mathcal{C}_1 f\|_{\Lambda_1^p(\mathbb{R}^+)} \le p' \|f\|_{\Lambda_1^p(\mathbb{R}^+)}
\end{equation}
for all $f \in \mathcal{M}^+(\mathbb{R}^+)$ and $1 < p \le \infty$. In the limiting case $p = \infty$ in \eqref{Equivalence1}, the constants in the lower and upper bounds are given by 
\[
\lim_{p \to \infty} \frac{\Gamma(p+1)^{1/p}}{p-1} = \frac{1}{e} \quad \text{and} \quad \lim_{p \to \infty} p' = 1,
\]
respectively. It is worth emphasizing that the constants appearing in \eqref{Equivalence1} are optimal. This line of research was initiated by K.~Le\'snik and L.~Maligranda in \cite{Mal} and subsequently completed by A.~Ben Said, S.~Boza, and J.~Soria in \cite{ABS3}.

\medskip

Observe that since $\mathcal{C}_1$ is a linear injective operator, the functional 
\[
N(f) = \|\mathcal{C}_1(|f|)\|_{\Lambda_1^p(\mathbb{R}^+)}
\]
defines a norm on $\Lambda_1^p(\mathbb{R}^+)$ that is equivalent to its standard norm for all $1<p\leq \infty$, in view of \eqref{Equivalence1}.

\medskip

In recent years, considerable attention has been devoted to the problem of renorming Ces\`aro function spaces. For instance, G. Bennett \cite[Theorem 21.1]{B} established that the spaces $\Lambda_1^p(\mathbb{R}^+)$ and $\Lambda_1^{*p}(\mathbb{R}^+)$ coincide for $1 < p < \infty$. In terms of norms, this equivalence ensures the existence of positive constants $c(p)$ and $C(p)$ such that
\[
c(p) \|f\|_{\Lambda_1^{*p}(\mathbb{R}^+)} \le \|f\|_{\Lambda_1^p(\mathbb{R}^+)} \le C(p) \|f\|_{\Lambda_1^{*p}(\mathbb{R}^+)}
\]
for all $f \in \mathcal{M}^+(\mathbb{R}^+)$ and $1 < p < \infty$. Bennett also provided explicit estimates for some of these equivalence constants. In particular, he proved that
\[
\|f\|_{\Lambda_1^{*p}(\mathbb{R}^+)} \le (p-1)^{1/p} \|f\|_{\Lambda_1^p(\mathbb{R}^+)}
\]
for $1 < p \le 2$, whereas the reverse inequality holds for $2 \le p < \infty$:
\[
(p-1)^{1/p} \|f\|_{\Lambda_1^p(\mathbb{R}^+)} \le \|f\|_{\Lambda_1^{*p}(\mathbb{R}^+)}.
\]
Later, V.~Kolyada completed the work started by G. Bennett established that:
\begin{theorem}{\rm \cite[Theorem 1.1]{Ko2}}
\label{THEOREMAKOLYADA}
Let $1<p<\infty$ and let $f\in \mathcal{M}(\mathbb{R}^+)$. If $1<p\leq 2$, then
\begin{equation} 
\label{K1} 
(p-1)\|f\|_{\Lambda_1^{p}(\mathbb{R}^+)} \leq \|f\|_{\Lambda_1^{*p}(\mathbb{R}^+)} \leq (p-1)^{1/p} \|f\|_{\Lambda_1^{p}(\mathbb{R}^+)}, 
\end{equation} 
 and if $2\leq p<\infty$, then
 \begin{equation}
  \label{K2} 
  (p-1)^{1/p}\|f\|_{\Lambda_1^{p}(\mathbb{R}^+)} \leq \|f\|_{\Lambda_1^{*p}(\mathbb{R}^+)}  \leq (p-1) \|f\|_{\Lambda_1^{p}(\mathbb{R}^+)}.
  \end{equation} 
  Moreover, all constants in \eqref{K1} and \eqref{K2} are the best possible. 
  \end{theorem}
  
\medskip

In \cite{ABS3}, A.~Ben Said, S.~Boza, and J.~Soria also proved the following theorem.

\begin{theorem}{\rm \cite[Theorem 2.2]{ABS3}}
\label{Alice100}
Let $1\leq p <\infty$ and let $f \in \mathcal{M}(\mathbb{R}^+)$. Then the following inequalities are sharp:
\begin{equation}
\label{A21}
\Gamma(p+1)^{1/p}\|f\|_{\Lambda_1^{*p}(\mathbb{R}^+)} \leq \|\mathcal{C}_1^*f\|_{\Lambda_1^{*p}(\mathbb{R}^+)}\leq p \|f\|_{\Lambda_1^{*p}(\mathbb{R}^+)}.
\end{equation}
\end{theorem}

\medskip

Let $C_p=\int_0^1|\ln(t)+1|^pdt$. As a consequence of a result initiated in \cite{ABS}, A. Ben Said and G. Sinnamon in \cite{AS} obtained the result below:
\begin{theorem}{\rm \cite[Theorem 1.2]{AS}}
Let $1<p<\infty$ and ~$f\in \mathcal{M}(\mathbb{R}^+)$ be such that $\mathcal{C}_1^*|f(x)|<\infty$ for all $x>0$. If $1<p \leq 2$, then
\begin{align}
\label{A1}
(p-1)\|f\|_{\Lambda_1^{*p}(\mathbb{R}^+)} &\leq\|(\mathcal{C}_1^*-I)\mathcal{C}_1^*|f|\|_{p}\leq C_p^{1/p}\|f\|_{\Lambda_1^{*p}(\mathbb{R}^+)},
\end{align}
and if $2\leq p\leq \infty$, then
\begin{align}
\label{A2}
C_p^{1/p}\|f\|_{\Lambda_1^{*p}(\mathbb{R}^+)}&\leq\|(\mathcal{C}_1^*-I)\mathcal{C}_1^*|f|\|_{p}\leq (p-1)\|f\|_{\Lambda_1^{*p}(\mathbb{R}^+)}.
\end{align}
The constants $p-1$ and $C_p^{1/p}$ are optimal in both \eqref{A1} and \eqref{A2}.
\end{theorem}

\medskip

For each $m > 0$,  natural generalizations of the Ces\`aro function spaces $\Lambda_1^p(\mathbb{R}^+)$ and the Copson function spaces $\Lambda_1^{*p}(\mathbb{R}^+)$ are the function spaces $\Lambda_m^p(\mathbb{R}^+)$ and $\Lambda_m^{*p}(\mathbb{R}^+)$, defined respectively by
\[
\Lambda_m^p(\mathbb{R}^+) = \bigl\{ f \in \mathcal{M}(\mathbb{R}^+) : \|f\|_{\Lambda_m^p(\mathbb{R}^+)} = \|\mathcal{C}_m|f|\|_p < \infty \bigr\}, \qquad 1 < p \le \infty,
\]
and
\[
\Lambda_m^{*p}(\mathbb{R}^+) = \bigl\{ f \in \mathcal{M}(\mathbb{R}^+) : \|f\|_{\Lambda_m^{*p}(\mathbb{R}^+)} = \|\mathcal{C}_m^*|f|\|_p < \infty \bigr\}, \qquad 1 \le p < \infty.
\]

\medskip

The main goal of this paper is to establish the existence of positive constants $c, C, k, K$ (depending only on $m, n \in \mathbb{N}$ and $p$) such that 
\begin{equation}
\label{IneqGoal1}
c \, \|f\|_{\Lambda_n^p(\mathbb{R}^+)} \le \|f\|_{\Lambda_m^p(\mathbb{R}^+)} \le C \, \|f\|_{\Lambda_n^p(\mathbb{R}^+)}
\end{equation}
and
\begin{equation}
\label{IneqGoal2}
k \, \|f\|_{\Lambda_n^{*p}(\mathbb{R}^+)} \le \|f\|_{\Lambda_m^{*p}(\mathbb{R}^+)} \le K \, \|f\|_{\Lambda_n^{*p}(\mathbb{R}^+)}
\end{equation}
for all $f \in \mathcal{M}(\mathbb{R}^+)$. Furthermore, we aim to determine the optimal values for $c, C, k,$ and $K$ in \eqref{IneqGoal1} and \eqref{IneqGoal2}.

Note that \eqref{IneqGoal1} and \eqref{IneqGoal2} imply, respectively, the space equalities
\begin{equation}
\label{Inclusion1}
\Lambda_m^p(\mathbb{R}^+) = \Lambda_n^p(\mathbb{R}^+) \quad \text{for } 1 < p \le \infty
\end{equation}
and
\begin{equation}
\label{Inclusion2}
\Lambda_m^{*p}(\mathbb{R}^+) = \Lambda_n^{*p}(\mathbb{R}^+) \quad \text{for } 1 \le p < \infty
\end{equation}
for any $m, n \in \mathbb{N}$. Combining \eqref{Inclusion1} and \eqref{Inclusion2} with Theorem~\ref{THEOREMAKOLYADA}, we deduce that
\[
\Lambda_m^p(\mathbb{R}^+) = \Lambda_n^p(\mathbb{R}^+) = \Lambda_r^{*p}(\mathbb{R}^+) = \Lambda_s^{*p}(\mathbb{R}^+)
\]
for all $1 < p < \infty$ and all $m, n, r, s \in \mathbb{N}$.

\medskip

This paper is organized as follows. Section~\ref{MainResults} presents our main results, whose formal proofs are detailed in Section~\ref{ProofsMainResults}. Section~\ref{Preliminaries} collects the necessary preliminary concepts and technical tools. Finally, Section~\ref{FurtherComments} summarizes all constants derived in this study in tabular form and poses a conjecture regarding those that remain open.

\section{Main Results}
\label{MainResults}

Let $B$ denote the Beta function, defined for $\operatorname{Re}(z_1), \operatorname{Re}(z_2) > 0$ by 
\[
B(z_1,z_2) = \int_0^1 v^{z_1-1} (1-v)^{z_2-1} \, dv = \frac{\Gamma(z_1)\Gamma(z_2)}{\Gamma(z_1+z_2)}.
\]

The following theorem establishes the continuous embeddings 
\[
\Lambda_n^p(\mathbb{R}^+) \hookrightarrow \Lambda_{n+k}^p(\mathbb{R}^+) \quad \text{for all } 1 < p \le \infty
\]
and
\[
\Lambda_n^{*p}(\mathbb{R}^+) \hookrightarrow \Lambda_{n+k}^{*p}(\mathbb{R}^+) \quad \text{for all } 1 \le p < \infty,
\]
for every $n, k \in \mathbb{N}$, while also determining the exact norms of these inclusion operators.
\begin{theorem}
\label{Theorem 3.1.}
Let $1 \leq  p \le \infty$, $n, k \in \mathbb{N}$, and $f \in \mathcal{M}^+(\mathbb{R}^+)$. If $1<p\leq \infty$, then
\begin{equation}
\label{RelationCmCn}
\|\mathcal{C}_{n+k}f\|_p \le \frac{B\big(n + \frac{1}{p'}, k\big)}{B(n+1, k)} \|\mathcal{C}_{n}f\|_p
\end{equation}
and if $1\leq p < \infty$, then
\begin{equation}
\label{RelationC*m*Cn}
\|\mathcal{C}^*_{n+k}f\|_p \le \frac{B\big(n + \frac{1}{p}, k\big)}{B(n+1, k)} \|\mathcal{C}^*_{n}f\|_p.
\end{equation}
Furthermore, the constants in \eqref{RelationCmCn} and \eqref{RelationC*m*Cn} are sharp.
\end{theorem}

For all $n, k \in \mathbb{N}$, the following theorem proves the continuous embedding 
\[
\Lambda_{n+k}^p(\mathbb{R}^+) \hookrightarrow \Lambda_{n}^p (\mathbb{R}^+)
\]
when $p=2$ or $p=\infty$, and explicitly determines the norms of the corresponding inclusion operators.
\begin{theorem}
\label{Theorem3.2.}
Let $n, k \in \mathbb{N}$ and let $f \in \mathcal{M}^+(\mathbb{R}^+)$. Then
\begin{equation}
\label{IneqCnInfty}
\|\mathcal{C}_n f\|_\infty \leq  \frac{(n-1)^{n-1}}{n^{n-1}} \frac{(n+k)^{n+k-1}}{(n+k-1)^{n+k-1}} \|\mathcal{C}_{n+k} f\|_{\infty}
\end{equation}
and
\begin{equation}
\label{C_nC_mInequality}
\|\mathcal{C}_n f \|_2^2 \leq \bigg( \frac{n}{n+k} \bigg)^2 \frac{B\big( 2(n-1)+1,1\big)}{B\big(2(n+k-1)+1,1\big)}  \|\mathcal{C}_{n+k} f \|_2^2 .
\end{equation}
Moreover, the constants in \eqref{IneqCnInfty} and in \eqref{C_nC_mInequality} are the best possible.
\end{theorem}

For all $n, k \in \mathbb{N}$, the following result proves the continuous embedding 
\[
\Lambda_{n+k}^{*p}(\mathbb{R}^+) \hookrightarrow \Lambda_{n}^{*p} (\mathbb{R}^+)
\]
when $p=1$ or $p=2$, and explicitly provides the norms of the corresponding inclusion operators.
\begin{theorem}
\label{Theorem3.6.}
Let $n,k \in \mathbb{N}$ and let $f \in \mathcal{M}^+(\mathbb{R}^+)$. If $p=1$ or $p=2$, then
\begin{equation}
\label{C*_nC_^*Inequality}
\|\mathcal{C}^*_n f \|_p^p \leq \frac{n^p}{(n+k)^p} \frac{p(n+k-1)+1}{p(n-1)+1}  \|\mathcal{C}^*_{n+k} f \|_p^p.
\end{equation}
Moreover, the constant in \eqref{C*_nC_^*Inequality} is optimal.
\end{theorem}

The following theorem demonstrates the continuous embedding 
\[
\Lambda_{2}^p(\mathbb{R}^+) \hookrightarrow \Lambda_{1}^p (\mathbb{R}^+)
\]
for all $1<p < \infty$ and explicitly determines the norm of the corresponding inclusion operator.
\begin{theorem}
\label{Theorem3.3.}
Let $1 < p < \infty$ and let $f \in \mathcal{M}^+(\mathbb{R}^+)$. Then
\begin{equation}
\label{IneqCnCn+1}
\|\mathcal{C}_1 f \|_p^p \leq  \frac{1}{2^p} \frac{B(1,p-1)}{B(p+1,p-1)} \|\mathcal{C}_{2} f \|_p^p.
\end{equation}
Moreover, the constant in \eqref{IneqCnCn+1} is sharp. 
\end{theorem}
The following theorem proves the continuous embedding 
\[
\Lambda_{2}^{*p}(\mathbb{R}^+) \hookrightarrow \Lambda_{1}^{*p} (\mathbb{R}^+)
\]
for all $1<p < \infty$ and explicitly determines the norm of the corresponding inclusion operator.
\begin{theorem}
\label{Theorem3.4.}
Let $1 < p < \infty$ and let $f \in \mathcal{M}^+(\mathbb{R}^+)$. Then
\begin{equation}
\label{IneqCn*Cn+1*}
\|\mathcal{C}_1^* f \|_p^p \leq  \frac{p+1}{2^p} \|\mathcal{C}_{2}^* f \|_p^p.
\end{equation}
Moreover, the constant in \eqref{IneqCn*Cn+1*} is the best possible. 
\end{theorem}

\section{Auxiliary results}
\label{Preliminaries} 

We begin by introducing an elementary lemma which will be used throughout the proofs of the main results in this paper.

\begin{lemma}
\label{ElementaryLemma}
Let $\{T_n\}_{n \ge 1}$ be a sequence of bounded operators on $L^p(\mathbb{R}^+)$, and let $\{C(n)\}_{n \ge 1}$ be a sequence of positive real numbers. If
\begin{equation*}
    \|T_n f\|_p \le \frac{C(n)}{C(n+1)} \|T_{n+1} f\|_p \quad \left(\text{resp. } \ge \right)
\end{equation*}
holds for each $n \in \mathbb{N}$ and all $f \in \mathcal{M}^+(\mathbb{R}^+)$, then
\begin{equation*}
    \|T_n f\|_p \le \frac{C(n)}{C(m)} \|T_m f\|_p \quad \left(\text{resp. } \ge \right)
\end{equation*}
holds for all $f \in \mathcal{M}^+(\mathbb{R}^+)$ and for all $m, n \in \mathbb{N}$ with $m \ge n$.
\end{lemma}

\begin{proof}
The result follows immediately by induction on $m \ge n$.
\end{proof}

The following result is a crucial lemma that relates $\mathcal{C}_m$ with $\mathcal{C}_{m+1}$ and $\mathcal{C}_m^*$ with $\mathcal{C}^*_{m+1}$.
\begin{lemma} Let $m \in \mathbb{N} \cup \{0\}$ and let $f \in \mathcal{M}^+(\mathbb{R}^+)$. Then
\begin{equation}
\label{Identity1}
\mathcal{C}_{m+1}f(x) = \frac{m+1}{x^{m+1}}\int_0^x t^m \mathcal{C}_{m} f(t)\,dt
\end{equation}
and
\begin{equation}
\label{Identity2}
\mathcal{C}^*_{m+1}f(x) = (m+1)x^m \int_x^\infty \frac{\mathcal{C}^*_m f(t)}{t^{m+1}}\,dt, 
\end{equation}
for all $x>0$.
\begin{proof}
 We first establish identity \eqref{Identity1}. The base case $m = 0$ is trivial. Now, let $m \ge 1$ be a natural number and let $f \in \mathcal{M}^+(\mathbb{R}^+)$. Applying Tonelli's theorem to exchange the order of integration, we obtain
\begin{align*}
    \frac{1}{x^{m+1}} \int_0^x t^m \mathcal{C}_m f(t) \, dt 
    &= \frac{m}{x^{m+1}} \int_0^x \int_0^t (t-u)^{m-1} f(u) \, du \, dt \\
    &= \frac{m}{x^{m+1}} \int_0^x f(u) \left( \int_u^x (t-u)^{m-1} \, dt \right) du \\
    &= \frac{1}{x^{m+1}} \int_0^x (x-u)^m f(u) \, du \\
    &= \frac{1}{m+1} \mathcal{C}_{m+1} f(x).
\end{align*}

\medskip

We now turn to identity \eqref{Identity2}, where the base case $m = 0$ is again immediate. For any natural number $m \ge 1$ and $f \in \mathcal{M}^+(\mathbb{R}^+)$, a similar argument using Tonelli's theorem and the change of variables $s = u/t$ can be followed to prove
\[
    (m+1)x^m \int_x^\infty \frac{\mathcal{C}^*_m f(t)}{t^{m+1}}\,dt = \mathcal{C}^*_{m+1}f(x).
\]
This concludes the proof
\end{proof}
\end{lemma}

 Throughout what follows, we adopt the standard convention $0^0 = 1$. We now introduce an important family of test functions.
\begin{lemma}
Let $1 < p \le \infty$ and let $m, n \in \mathbb{N}$. If $1<p<\infty$, then
\begin{equation}
\label{Limit1}
\lim_{\varepsilon \to 0^+} \frac{\| \mathcal{C}_m \chi_{(1,1+\varepsilon)}\|_p^p}{\| \mathcal{C}_{n} \chi_{(1,1+\varepsilon)}\|_p^p} = \bigg(\frac{m}{n}\bigg)^p \frac{B\big( (m-1)p + 1, p - 1 \big)}{B\big( (n-1)p + 1, p - 1 \big)}
\end{equation}
and
\begin{equation}
\label{Limit3}
\lim_{\varepsilon \to 0^+} \frac{\| \mathcal{C}^*_m \chi_{(1,1+\varepsilon)}\|_p^p}{\| \mathcal{C}^*_{n} \chi_{(1,1+\varepsilon)}\|_p^p} = \frac{m^p}{p(m-1)+1} \frac{p(n-1)+1}{n^p},
\end{equation}
If $p=\infty$, then
\begin{equation}
\label{Limit2}
\lim_{\varepsilon \to 0^+} \frac{\| \mathcal{C}_m \chi_{(1,1+\varepsilon)}\|_\infty}{\| \mathcal{C}_{n} \chi_{(1,1+\varepsilon)}\|_\infty} = \frac{(m-1)^{m-1}}{m^{m-1}} \frac{n^{n-1}}{(n-1)^{n-1}}.
\end{equation}
\begin{proof}
Let $m \in \mathbb{N}$ and $\varepsilon > 0$. We proceed to compute $\mathcal{C}_m \chi_{(1,1+\varepsilon)}$. For $0 < x \le 1$, we clearly have
\begin{align*}
\mathcal{C}_m \chi_{(1,1+\varepsilon)}(x) = \frac{m}{x^m} \int_0^x 0 \cdot (x-t)^{m-1} \, dt = 0.
\end{align*}
If $1 < x \le 1+\varepsilon$, it follows that
\begin{align*}
\mathcal{C}_m \chi_{(1,1+\varepsilon)}(x) = \frac{m}{x^m} \int_1^x (x-t)^{m-1} \, dt = \frac{(x-1)^m}{x^m},
\end{align*}
and for $x > 1+\varepsilon$, we obtain
\begin{align*}
\mathcal{C}_m \chi_{(1,1+\varepsilon)}(x) = \frac{m}{x^m} \int_1^{1+\varepsilon} (x-t)^{m-1} \, dt = \frac{(x-1)^m - (x-1-\varepsilon)^m}{x^m}.
\end{align*}
Therefore, for all $x > 0$, we have
\begin{equation}
\label{C_mf}
\mathcal{C}_m \chi_{(1,1+\varepsilon)}(x) = \frac{(x-1)^m}{x^m} \chi_{(1,1+\varepsilon]}(x) + \frac{(x-1)^m - (x-1-\varepsilon)^m}{x^m} \chi_{(1+\varepsilon,\infty)}(x).
\end{equation}
Taking the $L^p$ norm in \eqref{C_mf} we get that
\begin{align*}
\| \mathcal{C}_m \chi_{(1,1+\varepsilon)}\|_p^p &= \int_1^{1+\varepsilon} \frac{(x-1)^{mp}}{x^{mp}}\, dx+\int_{1+\varepsilon}^\infty \frac{\big((x-1)^m-(x-1-\varepsilon)^m\big)^p}{x^{mp}}dx \\
&=I_1(\varepsilon)+I_2(\varepsilon).
\end{align*}
On the one hand, for $x \in (1,1+\varepsilon)$, we have $1\leq x^{mp} \leq (1+\varepsilon)^{mp}$, so:
\[
\frac{1}{(1+\varepsilon)^{mp}} \int_1^{1+\varepsilon} (x-1)^{mp}\,dx \le I_1(\varepsilon)\leq \int_1^{1+\varepsilon} (x-1)^{mp}\,dx.
\]
Evaluating $\int_1^{1+\varepsilon} (x-1)^{mp}\,dx= \frac{\varepsilon^{mp+1}}{mp+1}$:
\[
\frac{\varepsilon^{mp+1}}{(1+\varepsilon)^{mp}(mp+1)}  \le I_1(\varepsilon) \leq\frac{\varepsilon^{mp+1}}{mp+1}.
\]
Divide the entire inequality chain by $\varepsilon^p$: 
\[
\frac{\varepsilon^{(m-1)p+1}}{(1+\varepsilon)^{mp}(mp+1)}  \le \frac{ I_1(\varepsilon)}{\varepsilon^p} \leq\frac{\varepsilon^{(m-1)p+1}}{mp+1}.
\]
Since $m\geq 1$ and $p>1$, the exponent $(m-1)p+1>0$. Taking $\varepsilon \to 0^+$ sends both the left and right bounds to $0$. By the Squeeze Theorem:
$$\lim_{\varepsilon \to 0^+} \frac{ I_1(\varepsilon)}{\varepsilon^p} = 0.$$
On the other hand, by applying the change of variables $v=1-\frac{1}{x}$, we get
\begin{align*}
I_2(\varepsilon)&=\int_{1+\varepsilon}^\infty \frac{\big((x-1)^m-(x-1-\varepsilon)^m\big)^p}{x^{mp}}\, dx\\
&=\int_{1+\varepsilon}^\infty \bigg(\bigg(1-\frac{1}{x}\bigg)^m-\bigg(1-\frac{1}{x}-\frac{\varepsilon}{x}\bigg)^m\bigg)^p\, dx \\
&=\int_{\frac{\varepsilon}{1+\varepsilon}}^1 \big(v^m-(v-\varepsilon(1-v))^m\big)^p \frac{dv}{(1-v)^2}.
\end{align*}
Dividing by $\varepsilon^p$ we obtain that
$$ \frac{I_2(\varepsilon)}{\varepsilon^p}=\int_{\frac{\varepsilon}{1+\varepsilon}}^1 \bigg(\frac{v^m-(v-\varepsilon(1-v))^m}{\varepsilon}\bigg)^p \frac{dv}{(1-v)^2}.$$
Now define the function inside the integral as $K(\varepsilon,v)$:
$$K(\varepsilon,v)=\chi_{(\frac{\varepsilon}{1+\varepsilon},1)}(v)\cdot\bigg(\frac{v^m-(v-\varepsilon(1-v))^m}{\varepsilon}\bigg)^p \frac{1}{(1-v)^2}.$$ 
Fix $v \in (0, 1)$ and define $g(y) = y^m$. By setting $h = \varepsilon(1-v)$, we observe that $h \to 0^+$ as $\varepsilon \to 0^+$. Expressing the difference quotient in terms of $g$, we have:
\begin{equation*}
\lim_{\varepsilon \to 0^+} \frac{v^m - (v - \varepsilon(1-v))^m}{\varepsilon(1-v)} = \lim_{h \to 0^+} \frac{g(v) - g(v-h)}{h} = g'(v) = mv^{m-1}.
\end{equation*}
This means
$$\lim_{\varepsilon \to 0^+} \frac{v^m - (v - \varepsilon(1-v))^m}{\varepsilon}=mv^{m-1}(1-v).$$
Thus
$$\lim_{\varepsilon \to 0^+} K(\varepsilon,v)=\big(mv^{m-1}(1-v)\big)^p\frac{1}{(1-v)^2}=m^pv^{(m-1)p}(1-v)^{p-2}.$$
For a fixed $v \in (0,1)$ and $\varepsilon > 0$, by the Mean Value Theorem, there exists a number $\xi \in (v - \varepsilon(1-v), v)$ such that
\begin{align*}
K(\varepsilon,v) &\le \left( \frac{v^m - (v - \varepsilon(1-v))^m}{\varepsilon} \right)^p \frac{1}{(1-v)^2} \\
&= \big( m \xi^{m-1} (1-v) \big)^p \frac{1}{(1-v)^2} \\
&\le \big( m v^{m-1} (1-v) \big)^p \frac{1}{(1-v)^2} \\
&= m^p v^{(m-1)p} (1-v)^{p-2} =: h(v).
\end{align*} 
Since $h \in L^1(0,1)$, we may apply the Dominated Convergence Theorem to conclude that
\begin{align*}
\lim_{\varepsilon \to 0^+} \frac{I_2(\varepsilon)}{\varepsilon^p} 
&= \lim_{\varepsilon \to 0^+} \int_0^1 K(\varepsilon,v) \, dv 
= \int_0^1 \lim_{\varepsilon \to 0^+} K(\varepsilon,v) \, dv \\
&= \int_0^1 m^p v^{(m-1)p} (1-v)^{p-2} \, dv 
= m^p B\big( (m-1)p + 1, p - 1 \big).
\end{align*}
Hence
\begin{align*}
\lim_{\varepsilon \to 0^+}  \frac{\| \mathcal{C}_m \chi_{(1,1+\varepsilon)}\|_p^p}{\varepsilon^p}  = \lim_{\varepsilon \to 0^+} \bigg(\frac{I_1(\varepsilon)}{\varepsilon^p}+\frac{I_2(\varepsilon)}{\varepsilon^p}\bigg)
&= m^p B\big( (m-1)p + 1, p - 1 \big),
\end{align*}
for all $m \in \mathbb{N}$. Therefore
\begin{align*}
\lim_{\varepsilon \to 0^+} \frac{\| \mathcal{C}_m \chi_{(1,1+\varepsilon)}\|_p^p}{\| \mathcal{C}_{n} \chi_{(1,1+\varepsilon)}\|_p^p} &=\lim_{\varepsilon \to 0^+}  \bigg(\frac{\| \mathcal{C}_m \chi_{(1,1+\varepsilon)}\|_p^p}{\varepsilon^p} \frac{\varepsilon^p}{\| \mathcal{C}_{n} \chi_{(1,1+\varepsilon)}\|_p^p}\bigg) \\
&= \bigg(\frac{m}{n}\bigg)^p \frac{B\big( (m-1)p + 1, p - 1 \big)}{B\big( (n-1)p + 1, p - 1 \big)},
\end{align*}
for all $m,n \in \mathbb{N}$.

\medskip

We now consider the case $p=\infty$ in order to prove \eqref{Limit2}. Suppose that $m \ge 2$. Recall from \eqref{C_mf} that
\begin{equation*}
\mathcal{C}_m \chi_{(1,1+\varepsilon)}(x) = \frac{(x-1)^m}{x^m} \chi_{(1,1+\varepsilon]}(x) + \frac{(x-1)^m - (x-1-\varepsilon)^m}{x^m} \chi_{(1+\varepsilon,\infty)}(x).
\end{equation*}
Observe that $x \mapsto \mathcal{C}_m \chi_{(1,1+\varepsilon)}(x)$ is continuous on $\mathbb{R}^+$ and increasing on $(0,1+\varepsilon)$. Then, let $f_\varepsilon: (1+\varepsilon,\infty) \to \mathbb{R}$ be defined by 
\[
f_\varepsilon(x) = \frac{(x-1)^m - (x-1-\varepsilon)^m}{x^m}.
\]
Setting $h_\varepsilon(x) = f_\varepsilon(1/x) = (1-x)^m - \big(1 - x(1+\varepsilon)\big)^m$ for $x \in \big(0, \frac{1}{1+\varepsilon}\big)$, we differentiate $h_\varepsilon$ with respect to $x$:
\[
h_\varepsilon'(x) = -m(1-x)^{m-1} + m(1+\varepsilon)\big(1 - x(1+\varepsilon)\big)^{m-1}.
\]
Thus,
\[
h_\varepsilon'(x) = 0 \iff (1+\varepsilon)^{\frac{1}{m-1}}\big(1 - x(1+\varepsilon)\big) = 1 - x,
\]
which yields
\[
x = \frac{(1+\varepsilon)^{\frac{1}{m-1}} - 1}{(1+\varepsilon)^{\frac{m}{m-1}} - 1}.
\]
Consequently, $f_\varepsilon$ has a unique critical point at
\[
x_0 = \frac{(1+\varepsilon)^{\frac{m}{m-1}} - 1}{(1+\varepsilon)^{\frac{1}{m-1}} - 1} \in (1+\varepsilon,\infty).
\]
A straightforward calculation gives
\[
f_\varepsilon(x_0) = \frac{\varepsilon^m}{\left( (1+\varepsilon)^{\frac{m}{m-1}} - 1 \right)^{m-1}}.
\]
Moreover, since the function $x \mapsto (1 - 1/x)^m$ is strictly increasing on $(1,1+\varepsilon)$, we have
\[
\left(1 - \frac{1}{x}\right)^m \le \left(1 - \frac{1}{1+\varepsilon}\right)^m = \frac{\varepsilon^m}{(1+\varepsilon)^m} < \frac{\varepsilon^m}{\left( (1+\varepsilon)^{\frac{m}{m-1}} - 1 \right)^{m-1}}
\]
for all $x \in (1,1+\varepsilon)$. Therefore, the maximum of $\mathcal{C}_m\chi_{(1,1+\varepsilon)}$ is achieved at $x_0$, so
\[
\|\mathcal{C}_m\chi_{(1,1+\varepsilon)}\|_\infty = \frac{\varepsilon^m}{\left( (1+\varepsilon)^{\frac{m}{m-1}} - 1 \right)^{m-1}}
\]
for all $\varepsilon > 0$. Taking the limit as $\varepsilon \to 0^+$, we obtain
\begin{align*}
\lim_{\varepsilon \to 0^+} \frac{\varepsilon}{\|\mathcal{C}_m\chi_{(1,1+\varepsilon)}\|_\infty} 
&= \lim_{\varepsilon \to 0^+} \left( \frac{(1+\varepsilon)^{\frac{m}{m-1}} - 1}{\varepsilon} \right)^{m-1} = \left( \frac{m}{m-1} \right)^{m-1}.
\end{align*}
This implies that
\begin{equation}
\label{LimitXX}
\lim_{\varepsilon \to 0^+} \frac{\|\mathcal{C}_m\chi_{(1,1+\varepsilon)}\|_\infty}{\varepsilon} = \frac{(m-1)^{m-1}}{m^{m-1}}
\end{equation}
for all $m \in \mathbb{N}$ with $m \ge 2$.

For the case $m=1$, by using \eqref{C_mf}, we have
\[
\mathcal{C}_1 \chi_{(1,1+\varepsilon)}(x) = \left(1-\frac{1}{x}\right) \chi_{(1,1+\varepsilon]}(x) + \frac{\varepsilon}{x} \chi_{(1+\varepsilon,\infty)}(x) \quad \text{for all } x > 0.
\]
The supremum is clearly attained at $x = 1+\varepsilon$, yielding
\[
\|\mathcal{C}_1 \chi_{(1,1+\varepsilon)}\|_\infty = \mathcal{C}_1 \chi_{(1,1+\varepsilon)}(1+\varepsilon) = \frac{\varepsilon}{1+\varepsilon}.
\]
Consequently,
\begin{equation}
\label{LimitXXX}
\lim_{\varepsilon \to 0^+} \frac{\| \mathcal{C}_1 \chi_{(1,1+\varepsilon)}\|_\infty}{\varepsilon} = \lim_{\varepsilon \to 0^+} \frac{1}{1+\varepsilon} = 1.
\end{equation}
Combining \eqref{LimitXX} and \eqref{LimitXXX} under the convention $0^0 = 1$, we conclude that
\begin{equation}
\label{LimitX}
\lim_{\varepsilon \to 0^+} \frac{\|\mathcal{C}_m\chi_{(1,1+\varepsilon)}\|_\infty}{\varepsilon} = \frac{(m-1)^{m-1}}{m^{m-1}}
\end{equation}
holds for all $m \in \mathbb{N}$ with $m \ge 1$.
Applying \eqref{LimitX}, we then obtain
\begin{align*}
\lim_{\varepsilon \to 0^+} \frac{\| \mathcal{C}_m \chi_{(1,1+\varepsilon)}\|_\infty}{\| \mathcal{C}_{n} \chi_{(1,1+\varepsilon)}\|_\infty} &=\lim_{\varepsilon \to 0^+}  \bigg(\frac{\| \mathcal{C}_m \chi_{(1,1+\varepsilon)}\|_\infty}{\varepsilon} \frac{\varepsilon}{\| \mathcal{C}_{n} \chi_{(1,1+\varepsilon)}\|_\infty}\bigg) = \frac{(m-1)^{m-1}}{m^{m-1}} \frac{n^{n-1}}{(n-1)^{n-1}}.
\end{align*}
for all $m,n \in \mathbb{N}$.

\medskip

We now proceed to establish the limit asserted in \eqref{Limit3}. A direct evaluation yields
\[
\mathcal{C}_m^* \chi_{(1,1+\epsilon)}(x) = m \int_1^{1+\varepsilon} \frac{(t-x)^{m-1}}{t^m} \, dt \cdot \chi_{(0,1]}(x) + m \int_x^{1+\varepsilon} \frac{(t-x)^{m-1}}{t^m} \, dt \cdot \chi_{(1,1+\varepsilon]}(x).
\]
Taking the $L^p$-norm, we can decompose the expression into two integrals, $I_1(\varepsilon)$ and $I_2(\varepsilon)$:
\begin{align*}
\|\mathcal{C}_m^*\chi_{(1,1+\varepsilon)}\|_p^p &= m^p \int_0^1 \left( \int_1^{1+\varepsilon} \frac{(t-x)^{m-1}}{t^m} \, dt \right)^p dx  + m^p \int_1^{1+\varepsilon} \left( \int_x^{1+\varepsilon} \frac{(t-x)^{m-1}}{t^m} \, dt \right)^p dx \\
&=: I_1(\varepsilon) + I_2(\varepsilon).
\end{align*}
To analyze $I_1(\varepsilon)$, define the auxiliary function $g_\varepsilon(x)$ on $(0,1)$ by
\[
g_\varepsilon(x) := \frac{m}{\varepsilon} \int_1^{1+\varepsilon} \frac{(t-x)^{m-1}}{t^m} \, dt.
\]
By the Lebesgue Differentiation Theorem, as $\varepsilon \to 0^+$, we have the pointwise convergence
\[
\lim_{\varepsilon \to 0^+} g_\varepsilon(x) = m(1-x)^{m-1} \quad \text{for all } x \in (0,1),
\]
which implies that
\[
\lim_{\varepsilon \to 0^+} (g_\varepsilon(x))^p = m^p(1-x)^{p(m-1)}.
\]
For $0 < \varepsilon < 1$, observe that $t^{-m} \le 1$ for $t \in (1, 1+\varepsilon)$. Since $t \le 1+\varepsilon < 2$, we have
\[
t - x \le 2 - x < 2 \quad \text{for all } x \in (0,1) \text{ and } t \in (1, 1+\varepsilon).
\]
This provides a uniform upper bound for the integrand:
\[
m \frac{(t-x)^{m-1}}{t^m} \le m (t-x)^{m-1} \le m (2-x)^{m-1}.
\]
Integrating over $(1, 1+\varepsilon)$ gives
\[
g_\varepsilon(x) \le m(2-x)^{m-1},
\]
and consequently,
\[
(g_\varepsilon(x))^p \le m^p(2-x)^{p(m-1)} =: w(x) \quad \text{for all } x \in (0,1).
\]
Since $w(x)$ is continuous on the compact interval $[0,1]$, then $w \in L^1(0,1)$. Applying the Dominated Convergence Theorem, we obtain
\begin{align*}
\lim_{\varepsilon \to 0^+} \frac{I_1(\varepsilon)}{\varepsilon^p} &= \lim_{\varepsilon \to 0^+} \int_0^1 (g_\varepsilon(x))^p \, dx = \int_0^1 \lim_{\varepsilon \to 0^+} (g_\varepsilon(x))^p \, dx \\
&= m^p \int_0^1 (1-x)^{p(m-1)} \, dx = \frac{m^p}{p(m-1)+1}.
\end{align*}

Next, to estimate $I_2(\varepsilon)$, note that for $x \in (1, 1+\varepsilon)$ and $t \in (x, 1+\varepsilon)$,
\[
\frac{m(t-x)^{m-1}}{t^m} \le m (t-x)^{m-1} \le m \varepsilon^{m-1}.
\]
Integrating this bound over $(x, 1+\varepsilon)$ yields
\[
\int_x^{1+\varepsilon} \frac{m(t-x)^{m-1}}{t^m} \, dt 
\le \int_x^{1+\varepsilon} m \varepsilon^{m-1} \, dt 
\le \int_1^{1+\varepsilon} m \varepsilon^{m-1} \, dt 
= m \varepsilon^m.
\]
Thus, $I_2(\varepsilon)/\varepsilon^p$ can be bounded as
\[
0 \le \frac{I_2(\varepsilon)}{\varepsilon^p} 
\le \frac{1}{\varepsilon^p} \int_1^{1+\varepsilon} m^p \varepsilon^{mp} \, dx 
= m^p \varepsilon^{(m-1)p+1}.
\]
Since $(m-1)p+1 > 0$, taking $\varepsilon \to 0^+$ and applying the Squeeze Theorem gives
\[
\lim_{\varepsilon \to 0^+} \frac{I_2(\varepsilon)}{\varepsilon^p} = 0.
\]
Combining our results for $I_1(\varepsilon)$ and $I_2(\varepsilon)$, we find that
\[
\lim_{\varepsilon \to 0^+} \frac{\|\mathcal{C}_m^*\chi_{(1,1+\varepsilon)}\|_p^p}{\varepsilon^p} 
= \lim_{\varepsilon \to 0^+} \frac{I_1(\varepsilon)}{\varepsilon^p} + \lim_{\varepsilon \to 0^+} \frac{I_2(\varepsilon)}{\varepsilon^p} 
= \frac{m^p}{p(m-1)+1}.
\]
Finally, taking the ratio for $m$ and $m+1$, we conclude that
\begin{align*}
\lim_{\varepsilon \to 0^+} \frac{\| \mathcal{C}^*_m \chi_{(1,1+\varepsilon)}\|_p^p}{\| \mathcal{C}^*_{n} \chi_{(1,1+\varepsilon)}\|_p^p} &=\lim_{\varepsilon \to 0^+}  \bigg(\frac{\| \mathcal{C}^*_m \chi_{(1,1+\varepsilon)}\|_p^p}{\varepsilon^p} \frac{\varepsilon^p}{\| \mathcal{C}^*_{n} \chi_{(1,1+\varepsilon)}\|_p^p}\bigg) \\
&= \frac{m^p}{p(m-1)+1} \frac{p(n-1)+1}{n^p},
\end{align*}
for all $m,n \in \mathbb{N}$. This completes the proof.
\end{proof}
\end{lemma}

The following lemma is of importance. 
\begin{lemma}
\label{LemmaPhi}
For any natural numbers $1 \le k \le m$, let $\Theta_{m,k} \colon (0,1] \to \mathbb{R}$ be defined by
\[
\Theta_{m,k}(t) = \frac{1}{2} \cdot \frac{t^{m-k} \int_0^t (1-s)^{m-1}(t-s)^{k-1} \, ds + \int_0^t (t-s)^{m-1}(1-s)^{k-1} \, ds}{\int_0^t (1-s)^{m-1}(t-s)^{m-1} \, ds}.
\]
Then $\Theta_{m,k}$ is non-increasing on $(0,1]$.
\begin{proof}
Applying the change of variables $s = rt$ (so that $ds = t\,dr$), we obtain
\[
\Theta_{m,k}(t) = \frac{1}{2} \frac{ \int_0^{1} (1-rt)^{m-1}(1-r)^{k-1}\, dr + \int_0^{1} (1-r)^{m-1}(1-rt)^{k-1}\, dr}{\int_0^{1}(1-rt)^{m-1}(1-r)^{m-1}\,dr}.
\]
In order to see that $\Theta_{m,k}$ is non-increasing on $(0,1]$ we will show that the function $\Phi_{m,k}$ defined by 
$$\Phi_{m,k}(t):=2\Theta_{m,k}(-t)$$ 
is non-decreasing on $[-1,0)$. We have that
\begin{align*}
\Phi_{m,k}(t) &= \frac{ \int_0^{1} (1+rt)^{m-1}(1-r)^{k-1}\, dr + \int_0^{1} (1-r)^{m-1}(1+rt)^{k-1}\, dr}{\int_0^{1}(1+rt)^{m-1}(1-r)^{m-1}\,dr}.
\end{align*}
An application of the binomial theorem followed by straightforward integration yields
$$\int_0^{1} (1+rt)^{m-1}(1-r)^{k-1}\, dr= \sum_{j=0}^{m-1} \frac{(m-1)!(k-1)!}{(m-1-j)!(k+j)!}t^j,$$
$$\int_0^{1} (1+rt)^{k-1}(1-r)^{m-1}\, dr= \sum_{j=0}^{k-1} \frac{(m-1)!(k-1)!}{(k-1-j)!(m+j)!}t^j,$$
and
$$\int_0^{1} (1+rt)^{m-1}(1-r)^{m-1}\, dr= \sum_{j=0}^{m-1} \frac{(m-1)!(m-1)!}{(m-1-j)!(m+j)!}t^j.$$
Let $P_{m,k}$ and $Q_{m,k}$ be the polynomials of degree $m-1$ defined by
\begin{align*}
P_{m,k}(t)&=\int_0^{1} (1+rt)^{m-1}(1-r)^{k-1}\, dr+ \int_0^{1} (1+rt)^{k-1}(1-r)^{m-1}\, dr \\
&=\sum_{j=0}^{k-1} \bigg(\frac{(m-1)!(k-1)!}{(m-1-j)!(k+j)!}+\frac{(m-1)!(k-1)!}{(k-1-j)!(m+j)!}\bigg)t^j \\
&\qquad + \sum_{j=k}^{m-1} \frac{(m-1)!(k-1)!}{(m-1-j)!(k+j)!}t^j=\sum_{j=0}^{m-1}\alpha_j t^j.
\end{align*}
\begin{align*}
Q_{m,k}(t)=\int_0^{1} (1+rt)^{m-1}(1-r)^{m-1}\, dr= \sum_{j=0}^{m-1} \frac{(m-1)!(m-1)!}{(m-1-j)!(m+j)!}t^j=\sum_{j=0}^{m-1}\beta_j t^j.
\end{align*}
We have that
\begin{align*}
\Phi_{m,k}(t)= \frac{P_{m,k}(t)}{Q_{m,k}(t)}=\frac{\sum_{j=0}^{m-1}\alpha_j t^j}{\sum_{j=0}^{m-1}\beta_j t^j}.
\end{align*}
If $k = m$, we trivially obtain that $\Phi_{m,m}(t) = 2$. In the case when $k=m-1$, we have that
\begin{align*}
P_{m,m-1}(t)&=(m-1)!(m-2)!\sum_{j=0}^{m-2} \bigg( \frac{1}{(m-1-j)!(m-1+j)!}+\frac{1}{(m-2-j)!(m+j)!}\bigg)t^j  \\
& \qquad + (m-1)!(m-2)! \frac{t^{m-1}}{(2m-2)!} \\
&=(m-1)!(m-2)!\sum_{j=0}^{m-2} \frac{2m-1}{(m-1-j)!(m+j)!}t^j+(m-1)!(m-2)! \frac{t^{m-1}}{(2m-2)!} \\
&=(m-1)!(m-2)!\sum_{j=0}^{m-1} \frac{2m-1}{(m-1-j)!(m+j)!}t^j \\
&= \frac{2m-1}{m-1} \sum_{j=0}^{m-1} \frac{(m-1)!(m-1)!}{(m-1-j)!(m+j)!}t^j= \frac{2m-1}{m-1} Q_{m,m-1}(t).
\end{align*}
Consequently, $\Phi_{m,m-1}(t) = \frac{P_{m,m-1}(t)}{Q_{m,m-1}(t)} = \frac{2m-1}{m-1}$ is a constant function, which is trivially non-decreasing on $[-1,0)$.

\medskip

Now assume that $1\leq k < m-1$. Let $A=\{A(j)\}_{j=0}^{m-1}$ be the vector defined by
\begin{align*}
A(j)=\frac{\alpha_j}{\beta_j}= \left\{ \begin{array}{lcc}
              \frac{(k-1)!}{(m-1)!} \left( \frac{(m+j)!}{(k+j)!} + \frac{(m-1-j)!}{(k-1-j)!} \right), &   {\rm if}  & 0 \leq j \leq k-1,  \\
              \\ \frac{(k-1)!}{(m-1)!} \frac{(m+j)!}{(k+j)!}, &  {\rm if}  & k \leq j \leq m-1.
              \end{array}
    \right.
\end{align*}
We are going to show that $A$ is an increasing vector. Let $0 \leq j \leq k-2$. 
\begin{align*}
A(j+1)-A(j)&= (m-k) \frac{(k-1)!}{(m-1)!} \left( \frac{(m+j)!}{(k+j+1)!} - \frac{(m-2-j)!}{(k-1-j)!} \right) \\
&=(m-k) \frac{(k-1)!}{(m-1)!} \left( T_1 - T_2 \right),
\end{align*}
where
\begin{align*}
T_1&=\frac{(m+j)!}{(k+j+1)!}=(m+j)(m+j-1)\cdots (k+j+2)\\
&=\prod_{i=1}^{m-k-1} (k+j+1+i)=\prod_{i=1}^{m-k-1} T_1(i)
\end{align*} 
and 
\begin{align*}
T_2&=\frac{(m-2-j)!}{(k-1-j)!}=(m-2-j)(m-3-j)\cdots (k-j) \\
&=\prod_{i=1}^{m-k-1} (k-j-1+i)=\prod_{i=1}^{m-k-1} T_2(i).
\end{align*}
Since $T_1(i) \geq T_2(i)$ for all $1 \leq i \leq m-k-1$, we get $T_1-T_2 \geq 0$. Therefore
\[
    A(j+1)-A(j)= (m-k) \frac{(k-1)!}{(m-1)!}\big(T_1-T_2) \geq 0, \quad \text{for all } 0 \leq j \leq k-2.
\]
For $j = k - 1$, similar reasoning as above yields:
\begin{align*}
A(k) - A(k-1) = \frac{(k-1)!(m-k)}{(m-1)!} \left(\frac{(m+k-1)!}{(2k)!} - (m-k-1)! \right) \geq 0.
\end{align*}
Finally, for $k \leq j \leq m-2$, we have that
$$A(j+1)-A(j)= \frac{(k-1)!(m-k)}{(m-1)!} \frac{(m+j)!}{(k+j+1)!} \geq 0.$$
Therefore
$$A(j+1)\geq A(j), \ \ \ {\rm for\  all}\  0 \leq j \leq m-2,$$
which means that $A$ is an increasing vector. Direct application of \cite[Theorem 4.4]{XX} implies that $\Phi_{m,k}$ is increasing on $[-1,0)$, which in turn yields that $\Theta_{m,k}$ is decreasing on $(0,1]$. This completes the proof.
\end{proof}
\end{lemma}

The following lemma will be needed.

\begin{lemma}
\label{LemaKernel*} Let $m \in \mathbb{N}$ and let $(u,v) \in \mathbb{R}^2$ be a non-negative vector. Then
{\small
\begin{align}
\label{KernelsPunctIneq}
 \frac{1}{2} \bigg(  \frac{1}{u^m} \frac{1}{v^k} \int_0^{\min\{u,v\}} & (u-x)^{m-1}(v-x)^{k-1}\, dx+ \frac{1}{v^m} \frac{1}{u^k} \int_0^{\min\{u,v\}} (v-x)^{m-1}(u-x)^{k-1}\, dx \bigg) \nonumber \\
  &\geq \frac{2m-1}{m+k-1}  \frac{1}{u^m}\frac{1}{v^m} \int_0^{\min\{u,v\}} (u-x)^{m-1}(v-x)^{m-1}\,dx,
\end{align}}
for all $1 \leq k \leq m$. 
\begin{proof}
Let $m \in \mathbb{N}$ and $1 \leq k \leq m$. For all $1 \leq j \leq m$, define
\[
K_{m,j}(u,v) = \frac{1}{2} \bigg( \frac{1}{u^m v^j} \int_0^{\min\{u,v\}} (u-x)^{m-1}(v-x)^{j-1}  dx+ \frac{1}{v^m u^j} \int_0^{\min\{u,v\}} (v-x)^{m-1}(u-x)^{j-1}  dx \bigg).
\]
Then, \eqref{KernelsPunctIneq} is equivalent to 
\begin{equation}
\label{KernelsPunctIneq_2}
K_{m,k}(u,v) \geq \frac{2m-1}{m+k-1} K_{m,m}(u,v),
\end{equation}
for all $u,v > 0$ and all $1 \leq k \leq m$. Since $K_{m,k}$ is symmetric in $u$ and $v$, without loss of generality we may assume that $v \leq u$. Let $t=\frac{v}{u} \in (0,1]$. Then \eqref{KernelsPunctIneq_2} is equivalent to 
\begin{align*}
 \frac{1}{2} \bigg(  \frac{1}{u^{m+k}} \frac{1}{t^k} \int_0^{tu} & (u-x)^{m-1}(tu-x)^{k-1}\, dx+ \frac{1}{t^m} \frac{1}{u^{m+k}} \int_0^{tu} (tu-x)^{m-1}(u-x)^{k-1}\, dx \bigg) \nonumber \\
  &\geq \frac{2m-1}{m+k-1}  \frac{1}{u^{2m}}\frac{1}{t^m} \int_0^{tu }(u-x)^{m-1}(tu-x)^{m-1}\,dx,
\end{align*}
for all $0<u<\infty$ and $t \in (0,1]$. Now, by doing the change of variable $x=us$, we get that \eqref{KernelsPunctIneq_2} is equivalent to prove that
\begin{align}
\label{KernelsPunctIneq_3}
 \frac{1}{2}  \frac{ t^{m-k}  \int_0^{t} (1-s)^{m-1}(t-s)^{k-1}\, ds+  \int_0^{t} (t-s)^{m-1}(1-s)^{k-1}\, ds}{\int_0^{t }(1-s)^{m-1}(t-s)^{m-1}\,ds}\geq \frac{2m-1}{m+k-1},
\end{align}
for all $t \in (0,1]$. Let $\Theta_{m,k}:(0,1] \to \mathbb{R}$ be the function given by
$$\Theta_{m,k}(t)= \frac{1}{2}  \frac{ t^{m-k}  \int_0^{t} (1-s)^{m-1}(t-s)^{k-1}\, ds+  \int_0^{t} (t-s)^{m-1}(1-s)^{k-1}\, ds}{\int_0^{t }(1-s)^{m-1}(t-s)^{m-1}\,ds}.$$
Note that 
\begin{align*}
\Theta_{m,k}(1)=\frac{\int_0^1 (1-s)^{m+k-2}\,ds }{\int_0^1 (1-s)^{2m-2}\,ds}=\frac{2m-1}{m+k-1}.
\end{align*}
Thus, to prove \eqref{KernelsPunctIneq_3}, it suffices to show that $\Theta_{m,k}$ is non-increasing on $(0,1]$, which is precisely the content of Lemma~\ref{LemmaPhi}.
\end{proof}
\end{lemma}

\section{Proofs of main results}
\label{ProofsMainResults}

Let $m, n \in \mathbb{N}$, and let $T_n$ and $T_m$ be two positive, bounded linear operators on $L^p(\mathbb{R}^+)$. Before presenting the main proofs, we note that any inequality of the form
\begin{equation}\label{General_Inequality}
\|T_n f\|_p \le C(n,m,p) \|T_m f\|_p
\end{equation}
needs only to be established for functions $f \in C_c^+(\mathbb{R}^+)$, the space of continuous, non-negative functions with compact support in $\mathbb{R}^+$. Indeed, for an arbitrary non-negative measurable function $f$, there exists a sequence $\{f_i\}_{i \in \mathbb{N}} \subset C_c^+(\mathbb{R}^+)$ such that 
\[
0 \le f_1(x) \le f_2(x) \le \cdots \leq f_i(x)\uparrow f(x) \quad \text{pointwise almost everywhere as } i \to \infty.
\]
Since $T_n$ and $T_m$ are positive operators, positivity implies monotonicity:
\[
0 \le T_n f_1(x) \le T_n f_2(x) \le \cdots \leq T_nf_i(x) \uparrow T_n f(x) \quad \text{pointwise almost everywhere as } i \to \infty.
\]
By the Monotone Convergence Theorem, for almost every $x > 0$:
\[
\lim_{i \to \infty} T_n f_i(x) = T_n f(x) \quad \text{and} \quad \lim_{i \to \infty} T_m f_i(x) = T_m f(x).
\]
Applying the Monotone Convergence Theorem once more to the increasing sequences $|T_n f_i(x)|^p \uparrow |T_n f(x)|^p$ and $|T_m f_i(x)|^p \uparrow |T_m f(x)|^p$, we obtain:
\[
\lim_{i \to \infty} \|T_n f_i\|_p = \|T_n f\|_p \quad \text{and} \quad \lim_{i \to \infty} \|T_m f_i\|_p = \|T_m f\|_p.
\]
Assuming that \eqref{General_Inequality} holds for each $f_i \in C_c^+(\mathbb{R}^+)$, taking $i \to \infty$ yields:
\[
\|T_n f\|_p = \lim_{i \to \infty} \|T_n f_i\|_p \le C(n,m,p) \lim_{i \to \infty} \|T_m f_i\|_p = C(n,m,p) \|T_m f\|_p.
\]
Therefore, the inequality
\[
\|\mathcal{C}_n f\|_p \le C(n,m,p) \|\mathcal{C}_m f\|_p
\]
holds for every non-negative measurable function $f$. If $\|\mathcal{C}_m f\|_p = \infty$, the right-hand side is infinite, making the inequality trivially true. If $\|\mathcal{C}_m f\|_p < \infty$, the inequality guarantees that $\|\mathcal{C}_n f\|_p$ is also finite and bounded appropriately.

\medskip

\noindent\textit{\textbf{Proof of Theorem \ref{Theorem 3.1.}.}}
We begin with inequality \eqref{RelationCmCn} in the case $n = m \in \mathbb{N}$ and $k = 1$. Let $f \in C_c^+(\mathbb{R}^+)$. We have seen in \eqref{Identity1} that 
\begin{equation*}
\mathcal{C}_{m+1}f(x) = \frac{m+1}{x^{m+1}}\int_0^x t^m \mathcal{C}_{m} f(t)\, dt,
\end{equation*}
for all $x>0$. By applying the change of variable $u = t/x$ in the identity above gives
\begin{align}
\label{Identity00}
    \mathcal{C}_{m+1}f(x) = (m+1) \int_0^1 u^m \mathcal{C}_{m} f(ux)\, du,
\end{align}
for all $x>0$. We start with the case $p=\infty$. Taking \eqref{Identity00} into account, we get the following.
\begin{align*}
\mathcal{C}_{m+1}f(x)  = (m+1) \int_0^1 u^m \mathcal{C}_{m} f(ux)\, du \leq \|\mathcal{C}_{m} f\|_{\infty} (m+1) \int_0^1 u^m \, du= \|\mathcal{C}_{m} f\|_{\infty},
\end{align*}
for all $x>0$. This implies 
$$\|\mathcal{C}_{m+1} f\|_{\infty} \leq \|\mathcal{C}_{m} f\|_{\infty},$$
for all $f \in C_c^+(\mathbb{R}^+)$. 

\medskip

Now, we turn to the case when $1<p<\infty$. Taking the $L^p$-norm on both sides of the identity \eqref{Identity00} and applying Minkowski's integral inequality, we find that
\begin{align*}
    \|\mathcal{C}_{m+1}f\|_p &= (m+1) \left( \int_0^\infty \left( \int_0^1 u^m \mathcal{C}_{m} f(ux)\, du \right)^p dx \right)^{1/p} \\
    &\leq (m+1) \int_0^1 \left(\int_0^\infty \big(u^m \mathcal{C}_m f(ux)\big)^p\,dx\right)^{1/p} du \\
    &= (m+1) \int_0^1 u^m \left( \int_0^\infty (\mathcal{C}_mf(ux))^p \,dx \right)^{1/p} du.
\end{align*}
Making the substitution $y = ux$ inside the inner integral (yielding $dx = u^{-1}dy$) decouples the variables, leading to
\begin{align*}
    \|\mathcal{C}_{m+1}f\|_p &\leq (m+1) \int_0^1 u^{m - 1/p} \,du \left( \int_0^\infty (\mathcal{C}_mf(y))^p \,dy \right)^{1/p} \\
    &= \frac{m+1}{m + 1 - 1/p} \|\mathcal{C}_mf\|_p = \frac{m+1}{m + 1/p'} \|\mathcal{C}_mf\|_p.
\end{align*}
Thus, we establish the norm recurrence relation
\begin{equation}
\label{INEQ}
\|\mathcal{C}_{m+1}f\|_p \leq \frac{\Gamma(m+2) \Gamma\left(m + \frac{1}{p'}\right)}{\Gamma(m+1) \Gamma\left(m + 1 + \frac{1}{p'}\right)} \|\mathcal{C}_{m}f\|_p=\frac{\frac{\Gamma\left(m + \frac{1}{p'}\right)}{\Gamma(m+1) }}{\frac{\Gamma\left(m + 1 + \frac{1}{p'}\right)}{\Gamma(m+2)}}\|\mathcal{C}_{m}f\|_p
\end{equation} 
for all $1<p\leq \infty$ and for all $f \in C_c^+(\mathbb{R}^+)$. Combining \eqref{INEQ} and Lemma~\ref{ElementaryLemma}, we obtain
\begin{align*}
    \|\mathcal{C}_{n+k}f\|_p  &\leq  \frac{\Gamma(n+k+1) \Gamma\left(n + \frac{1}{p'}\right)}{\Gamma(n+1) \Gamma\left(n+k + \frac{1}{p'}\right)} \|\mathcal{C}_{n}f\|_p=\frac{B\!\left(n + \frac{1}{p'},\, k\right)}{B(n+1,\, k)}\|\mathcal{C}_{n}f\|_p,
\end{align*}
for all $f \in C_c^+(\mathbb{R}^+)$.

To establish the sharpness of the constant in \eqref{RelationCmCn}, we proceed by contradiction. Suppose that there exists a constant $0 < C < B\big(n + \frac{1}{p'}, k\big)/B(n+1,\, k)$ such that
\[
\|\mathcal{C}_{n+k}f\|_p \le C \|\mathcal{C}_{n}f\|_p,
\]
for all $f \in \mathcal{M}^+(\mathbb{R}^+)$. Iterating inequality \eqref{INEQ} yields
\begin{align*}
\|\mathcal{C}_{n+k}|g|\|_p &\le  C\|\mathcal{C}_{n}|g|\|_p \le C \frac{\Gamma(n+1)\Gamma\left(n - 1 + \frac{1}{p'}\right)}{\Gamma(n)\Gamma\left(n + \frac{1}{p'}\right)} \|\mathcal{C}_{n-1}|g|\|_p \\
&\le \dots \le C \cdot \Gamma\left(\frac{1}{p'}\right) \frac{\Gamma(n+1)}{\Gamma\left(n + \frac{1}{p'}\right)} \|\mathcal{C}_0 |g|\|_p  \\
&= C \cdot \Gamma\left(\frac{1}{p'}\right) \frac{\Gamma(n+1)}{\Gamma\left(n + \frac{1}{p'}\right)} \|g\|_p,
\end{align*}
for all $g\in L^p(\mathbb{R}^+)$ with $1 < p \le \infty$. Consequently,
\[
\|\mathcal{C}_{n+k}g\|_p \le \|\mathcal{C}_{n+k}|g|\|_p \le C \cdot \Gamma\left(\frac{1}{p'}\right) \frac{\Gamma(n+1)}{\Gamma\left(n + \frac{1}{p'}\right)} \|g\|_p,
\]
for all $g\in L^p(\mathbb{R}^+)$ with $1 < p \le \infty$. This implies that the operator norm of $\mathcal{C}_{n+k}$ satisfies
\begin{align*}
\|\mathcal{C}_{n+k}\|_{L^p(\mathbb{R}^+)} 
&\le C \cdot \Gamma\left(\frac{1}{p'}\right) \frac{\Gamma(n+1)}{\Gamma\left(n + \frac{1}{p'}\right)} < \frac{\Gamma(n+k+1) \Gamma\left(n + \frac{1}{p'}\right)}{\Gamma(n+1) \Gamma\left(n+k + \frac{1}{p'}\right)} \cdot \Gamma\left(\frac{1}{p'}\right) \frac{\Gamma(n+1)}{\Gamma\left(n + \frac{1}{p'}\right)}\\& = \frac{\Gamma(n+k+1)\Gamma(1/p')}{\Gamma(n+k+1/p')},
\end{align*}
which contradicts the sharpness of inequality \eqref{BoundednessC}.

\medskip

We now turn to inequality \eqref{RelationC*m*Cn}. We begin with the case $p=1$. Let $f \in \mathcal{M}^+(\mathbb{R}^+)$. Applying identity \eqref{Identity2} and Tonelli's theorem, we obtain
\begin{align*}
\|\mathcal{C}_{m+1}^* f \|_1 &= \int_0^\infty \mathcal{C}^*_{m+1} f(x) \, dx = (m+1) \int_0^\infty x^m \int_x^\infty \frac{\mathcal{C}^*_m f(t)}{t^{m+1}} \, dt \, dx \\
&= (m+1) \int_0^\infty \frac{\mathcal{C}^*_m f(t)}{t^{m+1}} \left( \int_0^t x^m \, dx \right) dt = \int_0^\infty \mathcal{C}^*_m f(t) \, dt = \|\mathcal{C}_{m}^* f \|_1 
\end{align*}
for all $m \in \mathbb{N}$. Consequently, this yields
\begin{equation}
\label{C^*_1}
\|\mathcal{C}_{n+k}^* f \|_1 = \|\mathcal{C}_{n}^* f \|_1
\end{equation}
for all $n, k \in \mathbb{N}$. The proof of inequality \eqref{RelationC*m*Cn} for $1 < p < \infty$, along with its sharpness, proceeds in a manner completely analogous to that of \eqref{RelationCmCn}. One simply replaces the use of \eqref{Identity1} with \eqref{Identity2}, and substitutes \eqref{BoundednessC*} for \eqref{BoundednessC} when establishing optimality. This completes the proof. \hfill $\square$

\medskip

\noindent\textit{\textbf{Proof of Theorem \ref{Theorem3.2.}.}}
{\normalfont
We begin with inequality \eqref{IneqCnInfty} in the case $n = m \in \mathbb{N}$ and $k = 1$. The operators $\mathcal{C}_m$ and $\mathcal{C}_{m+1}$ can be expressed as integral operators by
\[
\mathcal{C}_{m}f(x) = \int_0^x f(t) K_m(x,t) \, dt \quad \text{for } x > 0,
\]
and
\[
\mathcal{C}_{m+1}f(u) = \int_0^u f(t) K_{m+1}(u,t) \, dt \quad \text{for } u > 0,
\]
where the respective kernels are given by
\[
K_m(x,t) = \frac{m(x-t)^{m-1}}{x^m} \quad \text{for } 0 < t \le x,
\]
and
\[
K_{m+1}(u,t) = \frac{(m+1)(u-t)^{m}}{u^{m+1}} \quad \text{for } 0 < t \le u.
\]
For a fixed $x > 0$, we set $u_x = \frac{m+1}{m}x$. Since $u_x > x$, the inclusion $(0,x) \subset (0,u_x)$ follows immediately. Let $R$ be the function defined on $(0,x)$ by
$$R(t)=\frac{K_m(x,t)}{K_{m+1}(u_x,t)}=\frac{m}{m+1} \frac{u_x^{m+1}}{x^m} \frac{(x-t)^{m-1}}{(u_x-t)^m}.$$
Substituting $u_x = \frac{m+1}{m}x$ we get
$$R(t)= \bigg(\frac{m+1}{m}\bigg)^m x \frac{(x-t)^{m-1}}{\big(\frac{m+1}{m}x-t\big)^m}=\bigg(\frac{m}{m+1}\bigg)^m x g(t),$$
where $g(t)=(x-t)^{m-1}\big(\frac{m+1}{m}x-t\big)^{-m}$. Now, taking derivative, and simple computations lead to
$$\bigg(\frac{m+1}{m}x-t\bigg)^{2m} \frac{d}{dt} g(t)= (x-t)^{m-2}\bigg(\frac{m+1}{m}x-t\bigg)^{m-1}\bigg(\frac{x}{m}-t\bigg).$$
Since $g'(t)>0$ for $0 < t< \frac{x}{m}$ and $g'(t)<0$ for $\frac{x}{m} <t<x$, we have that $g$ attains its unique maximum on $(0,x)$ at $t_0=\frac{x}{m}$. Evaluating $g(t_0)$ gives
$$g(t_0)= \frac{\big(\frac{m-1}{m}\big)^{m-1}x^{m-1}}{x^m}=\bigg(\frac{m-1}{m}\bigg)^{m-1}\frac{1}{x}.$$
Now, substituting $g(t_0)$ back into $R(t)$:
\begin{align*}
\underset{t \in (0,x)}{\max}R(t)=R(t_0)= \bigg(\frac{m+1}{m}\bigg)^m x\bigg(\frac{m-1}{m}\bigg)^{m-1}\frac{1}{x}=\frac{(m+1)^m(m-1)^{m-1}}{m^{2m-1}}:=L_m.
\end{align*}
This establishes that 
$$K_m(x,t) \leq L_m K_{m+1}\bigg(\frac{m+1}{m}x,t\bigg)$$
for all $t\in (0,x)$. Thus
\begin{align*}
\mathcal{C}_mf(x)&=\int_0^x f(t) K_m(x,t)\,dt \leq L_m \int_0^x f(t) K_{m+1}\bigg(\frac{m+1}{m}x,t\bigg)\,dt \\
&\leq L_m \int_0^{\frac{m+1}{m}x} f(t) K_{m+1}\bigg(\frac{m+1}{m}x,t\bigg)\,dt =L_m \mathcal{C}_{m+1}f\bigg(\frac{m+1}{m}x\bigg)\\ &\leq L_m \|\mathcal{C}_{m+1}f\|_\infty,
\end{align*}
for all $x>0$. Therefore
\begin{equation}
\label{InequalityY}
\|\mathcal{C}_{m}f\|_\infty \leq \frac{(m-1)^{m-1}}{m^{m-1}} \frac{(m+1)^m}{m^m} \|\mathcal{C}_{m+1}f\|_\infty.
\end{equation}
 This establishes \eqref{IneqCnInfty} for the case $n=m$ and $k=1$. Combining \eqref{InequalityY} and Lemma~\ref{ElementaryLemma}, we obtain
\begin{align*}
\|\mathcal{C}_{n}f\|_\infty \le \frac{(n-1)^{n-1}}{n^{n-1}} \frac{(n+k)^{n+k-1}}{(n+k-1)^{n+k-1}} \|\mathcal{C}_{n+k}f\|_\infty.
\end{align*}
This completes the proof of \eqref{IneqCnInfty}. The optimality of the constant in \eqref{IneqCnInfty} follows from \eqref{Limit2}.

\medskip

We now consider the case $p=2$. Since $\|\mathcal{C}_m f\|_2 = \|\mathcal{C}_m^* f\|_2$ for every $m \in \mathbb{N}$, applying inequality \eqref{C*_nC_^*Inequality} (the reader may check that its proof is independent) yields
\begin{align*}
\|\mathcal{C}_n f \|_2^2 &=\|\mathcal{C}_n^* f \|_2^2 \leq \frac{n^2}{2(n-1)+1} \frac{2(n+k-1)+1}{(n+k)^2} \|\mathcal{C}^*_{n+k} f \|_2^2 \\
&= \frac{n^2}{2(n-1)+1} \frac{2(n+k-1)+1}{(n+k)^2}  \|\mathcal{C}_{n+k} f \|_2^2 \\
&=\frac{n^2}{(n+k)^2}\frac{B\big( 2(n-1)+1,1\big)}{B\big(2(n+k-1)+1,1\big)}\|\mathcal{C}_{n+k} f \|_2^2.
\end{align*}
This completes the proof of inequality \eqref{C_nC_mInequality}. The optimality of the constant in \eqref{C_nC_mInequality} is an immediate consequence of \eqref{Limit1}. \hfill $\square$

\medskip

\noindent\textit{\textbf{Proof of Theorem \ref{Theorem3.6.}.}} The case $p=1$ is a direct consequence of the identity \eqref{C^*_1}. Let us now turn to the case $p=2$. We start proving inequality \eqref{C*_nC_^*Inequality} when $k=1$. We have to show that
\begin{equation}
\label{Ineq_p=2}
\|\mathcal{C}_m^* f \|_2^2 \leq \frac{m^2}{2(m-1)+1} \frac{2m+1} {(m+1)^2} \|\mathcal{C}_{m+1}^* f \|_2^2,
\end{equation}
for all $f \in C_c^+(\mathbb{R}^+)$ and for all $m \in \mathbb{N}$.
 Let $m \in \mathbb{N}$ and let $f \in C_c^+(\mathbb{R}^+)$. Then, applying integration by parts and differentiation under the integral sign yields
\begin{align*}
(m+1&)^{-2} \|\mathcal{C}^*_{m+1}f\|_2^2 = \int_0^\infty \bigg(\int_x^\infty \frac{f(t)}{t^{m+1}}(t-x)^m\, dt\bigg)^2\,dx \\
&= \int_0^\infty x^{2m} \bigg(\int_x^\infty \frac{f(t)}{t^{m+1}}\bigg(\frac{t}{x}-1\bigg)^m\, dt\bigg)^2\,dx\\
&=\frac{2m}{2m+1} \int_0^\infty x^{2m+1} \bigg(\int_x^\infty \frac{f(t)}{t^{m+1}}\bigg(\frac{t}{x}-1\bigg)^m dt\bigg)\bigg(\int_x^\infty \frac{f(u)}{u^{m+1}}\bigg(\frac{u}{x}-1\bigg)^{m-1}\frac{u}{x^2} du\bigg) dx \\
&=\frac{2m}{2m+1} \int_0^\infty  \bigg(\int_x^\infty \frac{f(t)}{t^{m+1}}(t-x)^m dt\bigg)\bigg(\int_x^\infty \frac{f(u)}{u^{m}}(u-x)^{m-1}\, du\bigg) dx.
\end{align*}
Using \eqref{Identity2}, we obtain
\begin{align*}
\int_x^\infty \frac{f(t)}{t^{m+1}}(t-x)^m\, dt&= \frac{\mathcal{C}_{m+1}^* f(x)}{m+1}=x^m \int_x^\infty \frac{\mathcal{C}^*_m f(t)}{t^{m+1}}\,dt\\
&=m x^m \int_x^\infty \frac{1}{t^{m+1}} \bigg(\int_t^\infty \frac{f(s)}{s^m}(s-t)^{m-1}\,ds \bigg)\,dt,
\end{align*}
for all $x>0$. Hence,
\begin{align*}
(m+1&)^{-2} \|\mathcal{C}^*_{m+1}f\|_2^2\\
&=\frac{2m^2}{2m+1} \int_0^\infty x^m \int_x^\infty \frac{1}{t^{m+1}} \int_t^\infty \frac{f(s)}{s^m}(s-t)^{m-1}ds \,dt\bigg(\int_x^\infty \frac{f(u)}{u^{m}}(u-x)^{m-1} du\bigg) dx.
\end{align*}
It follows from Tonelli's theorem that
\begin{align*}
(m+1&)^{-2} \|\mathcal{C}^*_{m+1}f\|_2^2\\
&=\frac{2m^2}{2m+1} \int_0^\infty \bigg(\int_t^\infty \frac{f(s)}{s^m}(s-t)^{m-1}\,ds \bigg) \frac{1}{t^{m+1}} \int_0^t x^m \int_x^\infty \frac{f(u)}{u^{m}}(u-x)^{m-1}\, du \,dx\, dt \\
&\geq \frac{2m^2}{2m+1} \int_0^\infty \bigg(\int_t^\infty \frac{f(s)}{s^m}(s-t)^{m-1}\,ds \bigg) \frac{1}{t^{m+1}} \int_0^t x^m \int_t^\infty \frac{f(u)}{u^{m}}(u-x)^{m-1}\, du \,dx\, dt.
\end{align*}
Now, for any $t > 0$, using the change of variables $v =x/t$ (so that $dv = dx/t$), we have
\begin{align*}
\frac{1}{t^{m+1}} \int_0^t x^m \int_t^\infty \frac{f(u)}{u^{m}}(u-x)^{m-1}\, du \,dx&= \frac{1}{t^{m+1}} \int_t^\infty \frac{f(u)}{u^m} \int_0^t x^m (u-x)^{m-1} \,dx \,du \\
& =  \int_t^\infty \frac{f(u)}{u^m} \int_0^1 v^m (u-tv)^{m-1} \, dv\,du \\
&= \int_t^\infty \frac{f(u)}{u} \int_0^1 v^m \bigg(1-\frac{t}{u}v\bigg)^{m-1} \, dv\,du.
\end{align*}
Evaluating the elementary integrals in terms of the Beta function yields
\begin{align*}
 \int_0^1 v^m \bigg(1-\frac{t}{u}v\bigg)^{m-1} \, dv & = \int_0^1 v^m \bigg(1-v+\bigg(1-\frac{t}{u}\bigg)v\bigg)^{m-1} \, dv \\
 &= \sum_{j=0}^{m-1} \binom{m-1}{j} \bigg(1-\frac{t}{u}\bigg)^j \int_0^1 v^{m+j}(1-v)^{m-1-j} dv \\
 &=\sum_{j=0}^{m-1} \frac{(m-1)!(m+j)!}{j! (2m)!}\bigg(1-\frac{t}{u}\bigg)^j.
\end{align*}
Thus,
\begin{align*}
\frac{1}{t^{m+1}} \int_0^t x^m \int_t^\infty \frac{f(u)}{u^{m}}(u-x)^{m-1}\, du \,dx= \sum_{j=0}^{m-1} \frac{(m-1)!(m+j)!}{j! (2m)!} \int_t^\infty \frac{f(u)}{u^{j+1}}(u-t)^j \,du,
\end{align*}
for all $t > 0$. Therefore,
\begin{align*}
(m+1&)^{-2} \|\mathcal{C}^*_{m+1}f\|_2^2\\
&\geq \frac{2m^2}{2m+1} \int_0^\infty \bigg(\int_t^\infty \frac{f(s)}{s^m}(s-t)^{m-1}\,ds \bigg) \frac{1}{t^{m+1}} \int_0^t x^m \int_t^\infty \frac{f(u)}{u^{m}}(u-x)^{m-1}\, du \,dx\, dt \\
&= \frac{2m^2}{2m+1}  \sum_{j=0}^{m-1} \frac{(m-1)!(m+j)!}{j! (2m)!}  \int_0^\infty \int_t^\infty \frac{f(s)}{s^m}(s-t)^{m-1}\,ds\int_t^\infty \frac{f(u)}{u^{j+1}}(u-t)^j \,du\,dt \\
&=\frac{1}{2m+1}  \sum_{j=0}^{m-1} \frac{m!(m+j)!}{j! (2m-1)!}  \int_0^\infty \int_t^\infty \frac{f(s)}{s^m}(s-t)^{m-1}\,ds\int_t^\infty \frac{f(u)}{u^{j+1}}(u-t)^j \,du\,dt.\\
&=\frac{1}{2m+1}  \sum_{j=0}^{m-1} \frac{m!(m+j)!}{j! (2m-1)!} \int_0^\infty \int_0^\infty \frac{f(s)}{s^m}\frac{f(u)}{u^{j+1}}\int_0^{\min\{u,s\}} (s-t)^{m-1}(u-t)^j dtduds \\
&=\frac{1}{2(2m+1)} \sum_{j=0}^{m-1} \frac{m!(m+j)!}{j! (2m-1)!}  \Bigg(\int_0^\infty \int_0^\infty \frac{f(s)}{s^m}\frac{f(u)}{u^{j+1}}\int_0^{\min\{u,s\}} (s-t)^{m-1}(u-t)^j dtduds \\
&\qquad + \int_0^\infty \int_0^\infty \frac{f(u)}{u^m}\frac{f(s)}{s^{j+1}}\int_0^{\min\{u,s\}} (u-t)^{m-1}(s-t)^j dtduds \Bigg)\\
&\geq \frac{2m-1}{2m+1} \sum_{j=0}^{m-1} \frac{m!(m+j)!}{j! (2m-1)!} \frac{1}{m+j} \int_0^\infty \int_0^\infty \frac{f(u)}{u^m}\frac{f(s)}{s^m} \int_0^{\min\{u,s\}} (u-x)^{m-1}(s-x)^{m-1}dtduds
\end{align*}
where we have applied Lemma~\ref{LemaKernel*} in the last inequality. Note that
\begin{align*}
\sum_{j=0}^{m-1} \frac{m!(m+j)!}{j! (2m-1)!} \frac{1}{m+j} &= \frac{m! (m-1)!}{(2m-1)!} \sum_{j=0}^{m-1} \frac{(m+j-1)!}{j!(m-1)!} = \frac{1}{\binom{2m-1}{m}} \sum_{j=0}^{m-1}\binom{m+j-1}{j} \\
&=\frac{1}{\binom{2m-1}{m}} \binom{2m-1}{m-1} = \frac{1}{\binom{2m-1}{m}} \binom{2m-1}{m}=1.
\end{align*}
Therefore,
\begin{align*}
(m+1)^{-2} \|\mathcal{C}^*_{m+1}f\|_2^2 & \geq \frac{2m-1}{2m+1} \int_0^\infty \int_0^\infty \frac{f(u)}{u^m}\frac{f(s)}{s^m} \int_0^{\min\{u,s\}} (u-x)^{m-1}(s-x)^{m-1}dtduds \\
& = \frac{2m-1}{2m+1} m^{-2} \|\mathcal{C}_m^* f\|_2^2.
\end{align*}
That is 
\begin{equation}
    \label{ImportanteProposition}
  \|\mathcal{C}_m^* f\|_2^2 \leq \bigg(\frac{m}{m+1}\bigg)^2  \frac{2m+1}{2(m-1)+1}  \|\mathcal{C}^*_{m+1}f\|_2^2.
\end{equation}
Combining \eqref{ImportanteProposition} and Lemma~\ref{ElementaryLemma}, we obtain
\begin{align*}
\|\mathcal{C}_n^* f \|_2^2 &\leq  \frac{n^2}{2(n-1)+1} \frac{2(n+k-1)+1}{(n+k)^2}  \|\mathcal{C}^*_{n+k} f \|_2^2.
\end{align*}
This proves inequality~\eqref{C*_nC_^*Inequality} when $p=2$. The proof that the constant in \eqref{C*_nC_^*Inequality} is sharp follows directly from \eqref{Limit3}.
\hfill $\square$

\medskip

\noindent\textit{\textbf{Proof of Theorem \ref{Theorem3.3.}.}}
Let $1 < p < \infty$ and let $f \in C_c^+(\mathbb{R}^+)$. We start by assuming that $1<p<2$. By using integration by parts we obtain that 
\begin{align*}
&2^{-p} \|\mathcal{C}_{2}f\|_p^p = \int_0^\infty x^{-2p} \bigg(\int_0^x f(t)(x-t)\,dt\bigg)^p\, dx\\
& =\frac{p}{2p-1}\int_0^\infty \frac{1}{x^{2p-1}}\bigg(\int_0^x f(t)(x-t)\,dt\bigg)^{p-1}\bigg(\int_0^x f(u)\,du\bigg)\, dx. 
\end{align*}
By applying Jensen's inequality, we get that
\begin{align*}
&\bigg(\int_0^x f(t)(x-t)\,dt\bigg )^{p-1} \geq \bigg(\int_0^x f(u)\,du\bigg)^{p-2}\bigg(\int_0^x f(t)(x-t)^{p-1}\,dt\bigg),
\end{align*}
for all $x>0$. Now, by using integration by parts we obtain that
\begin{align*}
\int_0^x f(t)(x-t)^{p-1}\,dt=(p-1) \int_0^x \bigg(\int_0^t f(w)\,dw\bigg)(x-t)^{p-2}\,dt,
\end{align*}
for all $x>0$. So
\[
\bigg(\int_0^x f(t)(x-t)\,dt\bigg )^{p-1} \geq (p-1) \bigg(\int_0^x f(u)\,du\bigg)^{p-2}\int_0^x \bigg(\int_0^t f(w)\,dw\bigg)(x-t)^{p-2}\,dt,
\]
for all $x>0$. Therefore
\begin{align*}
 2^{-p} \|\mathcal{C}_{2}f\|_p^p & = \frac{p}{2p-1}\int_0^\infty \frac{1}{x^{2p-1}}\bigg(\int_0^x f(t)(x-t)\,dt\bigg)^{p-1}\bigg(\int_0^x f(u)\,du\bigg)\, dx\\
&  \geq \frac{p(p-1)}{2p-1} \int_0^\infty \frac{1}{x^{2p-1}} \bigg(\int_0^x f(u)\,du\bigg)^{p-1}\int_0^x \bigg(\int_0^t f(w)\,dw\bigg)(x-t)^{p-2}\,dt \,dx \\
&  = \frac{p(p-1)}{2p-1}  \int_0^\infty \bigg(\int_0^t f(w)\,dw\bigg) \int_t^\infty \bigg(\int_0^x f(u)\,du\bigg)^{p-1} \frac{(x-t)^{p-2}}{x^{2p-1}}\,dx \, dt \\
&  \geq \frac{p(p-1)}{2p-1}  \int_0^\infty \bigg(\int_0^t f(w)\,dw\bigg)^p \int_t^\infty \frac{(x-t)^{p-2}}{x^{2p-1}}\,dx \, dt \\
&  = \frac{p(p-1)}{2p-1}  \int_0^\infty \bigg(\frac{1}{t}\int_0^t f(w)\,dw\bigg)^p \int_t^\infty \frac{(x/t-1)^{p-2}}{(x/t)^{2p-1}}\,\frac{dx}{t} \, dt \\
&  = \frac{p(p-1)}{2p-1}  \int_0^\infty \bigg(\frac{1}{t}\int_0^t f(w)\,dw\bigg)^p \int_1^\infty \frac{(u-1)^{p-2}}{u^{2p-1}}\,du \, dt \\
&  = \frac{p(p-1)}{2p-1} B(p,p-1) \int_0^\infty \bigg(\frac{1}{t}\int_0^t f(w)\,dw\bigg)^p  dt\\
&  = \frac{p(p-1)}{2p-1} \frac{\Gamma(p)\Gamma(p-1)}{\Gamma(2p-1)} \|\mathcal{C}_1f\|_p^p  \\
&   = \frac{B(p+1,p-1)}{B(1,p-1)}\|\mathcal{C}_1f\|_p^p,
\end{align*}
where we have applied of Tonelli's theorem again in the first equality and a change of variable $u=x/t$ in the third equality. Hence
$$\|\mathcal{C}_1 f \|_p^p \leq  \frac{1}{2^p} \frac{B(1,p-1)}{B(p+1,p-1)} \|\mathcal{C}_{2} f \|_p^p,$$
for all $1<p<2$ and for all $f \in C_c^+(\mathbb{R}^+)$. 

\medskip

Now we turn to the case $2 \leq p < \infty$. Let $ x>0$ and let $\varphi_x:(0,x] \to \mathbb{R}^+$ be the function defined by 
$$ \varphi_x(v)= \bigg( \int_0^v f(t)(x-t)\,dt\bigg)^{p-1}.$$ 
By applying the Fundamental Theorem of Calculus we obtain that
\begin{align*}
\bigg( \int_0^x f(t)(x-t)\,dt\bigg)^{p-1} &=\varphi_x(x) =\varphi_x(x)-\varphi_x(0)=\int_0^x \varphi_x'(v)\,dv \\
& = (p-1) \int_0^x \bigg( \int_0^v f(t)(x-t)\,dt\bigg)^{p-2}f(v)(x-v)\, dv \\
& \geq (p-1) \int_0^x \bigg( \int_0^v f(t)\,dt\bigg)^{p-2}f(v)(x-v)^{p-1}\, dv,
\end{align*}
for all $x>0$. Therefore
\begin{align*}
2^{-p} \|\mathcal{C}_{2}f\|_p^p &=\frac{p}{2p-1}\int_0^\infty \frac{1}{x^{2p-1}}\bigg(\int_0^x f(u)\,du\bigg)\bigg(\int_0^x f(t)(x-t)\,dt\bigg)^{p-1}\, dx \\
&  \geq \frac{p(p-1)}{2p-1}\int_0^\infty \frac{1}{x^{2p-1}}\bigg(\int_0^x f(u)\,du\bigg)\int_0^x \bigg( \int_0^v f(t)\,dt\bigg)^{p-2}f(v)(x-v)^{p-1}\, dv\, dx \\
& =\frac{p(p-1)}{2p-1} \int_0^\infty f(v) \bigg( \int_0^v f(t)\,dt\bigg)^{p-2} \int_v^\infty \frac{(x-v)^{p-1}}{x^{2p-1}}\bigg(\int_0^x f(u)\,du\bigg) \,dx \,dv \\
& = \frac{p(p-1)}{2p-1} \int_0^\infty  \bigg( \int_0^v f(t)\,dt\bigg)^{p-1} \int_v^\infty \frac{(x-v)^{p-2}}{x^{2p-1}}\bigg(\int_0^x f(u)\,du\bigg) \,dx \,dv \\
&  \geq \frac{p(p-1)}{2p-1} \int_0^\infty  \bigg( \int_0^v f(t)\,dt\bigg)^{p} \int_v^\infty \frac{(x-v)^{p-2}}{x^{2p-1}} \,dx \,dv \\
&  = \frac{B(p+1,p-1)}{B(1,p-1)}\|\mathcal{C}_1f\|_p^p,
\end{align*}
where we have used integration by parts (taking $ds = f(v) \big(\int_0^v f(t)dt\big)^{p-2} dv$) in the third equality. Thus
$$\|\mathcal{C}_1 f \|_p^p \leq  \frac{1}{2^p} \frac{B(1,p-1)}{B(p+1,p-1)} \|\mathcal{C}_{2} f \|_p^p,$$
for all $2 \leq p < \infty$ and for all $f \in C_c^+(\mathbb{R}^+)$.  This completes the proof of the inequality in \eqref{IneqCnCn+1}. The proof of the optimality of the constant in \eqref{IneqCnCn+1} follows from~\eqref{Limit1}.
\hfill $\square$

\medskip

\noindent\textit{\textbf{Proof of Theorem \ref{Theorem3.4.}.}}
Let $1 < p < \infty$ and let $f \in C_c^+(\mathbb{R}^+)$. We start by assuming that $1<p<2$. Integration by parts leads to
\begin{align*}
2^{-p} \|\mathcal{C}_2^*f \|_p^p &= \int_0^\infty \bigg(\int_x^\infty f(t) (t-x) \frac{dt}{t^2}\bigg)^p dx= \int_0^\infty x^p \bigg(\frac{1}{x}\int_x^\infty f(t) (t-x) \frac{dt}{t^2}\bigg)^p dx \\
& =\frac{p}{p+1} \int_0^\infty x^{p+1} \bigg(\frac{1}{x}\int_x^\infty f(t) (t-x) \frac{dt}{t^2}\bigg)^{p-1} \bigg(\frac{1}{x^2}\int_x^\infty \frac{f(v)}{v}\, dv\bigg)\,dx \\
& = \frac{p}{p+1} \int_0^\infty \bigg(\int_x^\infty \frac{f(t)}{t} \bigg(1-\frac{x}{t}\bigg)\, dt \bigg)^{p-1}\bigg(\int_x^\infty \frac{f(v)}{v} \,dv\bigg)\,dx.
\end{align*}
Now, applying Jensen's inequality, we have that
\begin{align*}
\bigg(\int_x^\infty \frac{f(t)}{t} \bigg(1-\frac{x}{t}\bigg)\, dt \bigg)^{p-1} \geq \bigg (\int_x^\infty \frac{f(v)}{v}\,dv\bigg)^{p-2} \int_x^\infty \frac{f(t)}{t} \bigg(1-\frac{x}{t}\bigg)^{p-1}\, dt,
\end{align*}
for all $x>0$. Therefore 
\begin{align*}
2^{-p} \|\mathcal{C}_2^*f \|_p^p & = \frac{p}{p+1} \int_0^\infty \bigg(\int_x^\infty \frac{f(v)}{v} \,dv\bigg) \bigg(\int_x^\infty \frac{f(t)}{t} \bigg(1-\frac{x}{t}\bigg)\, dt \bigg)^{p-1}\,dx\\ 
& \geq \frac{p}{p+1}\int_0^\infty \bigg(\int_x^\infty \frac{f(v)}{v}\,dv\bigg)^{p-1} \int_x^\infty \frac{f(t)}{t} \bigg(1-\frac{x}{t}\bigg)^{p-1}\, dt \, dx \\
& = \frac{p}{p+1} \int_0^\infty \frac{f(t)}{t} \int_0^t \bigg(\int_x^\infty \frac{f(v)}{v}\,dv\bigg)^{p-1} \bigg(1-\frac{x}{t}\bigg)^{p-1}\,dx\,dt \\
& =\frac{p(p-1)}{p+1} \int_0^\infty \bigg(\int_t^\infty \frac{f(u)}{u}\,du\bigg) \int_0^t \bigg(\int_x^\infty \frac{f(v)}{v}\,dv\bigg)^{p-1} \bigg(1-\frac{x}{t}\bigg)^{p-2}\frac{x}{t^2}\,dx\,dt \\
& \geq \frac{p(p-1)}{p+1} \int_0^\infty \bigg(\int_t^\infty \frac{f(u)}{u}\,du\bigg)^p \int_0^t  \bigg(1-\frac{x}{t}\bigg)^{p-2}\frac{x}{t}\,\frac{dx}{t}\,dt \\
& = \frac{p(p-1)}{p+1} \int_0^\infty \bigg(\int_t^\infty \frac{f(u)}{u}\,du\bigg)^p \int_0^1  (1-u)^{p-2}u\, du \,dt \\
& = \frac{p(p-1)}{p+1} B(p-1,2) \int_0^\infty \bigg(\int_t^\infty \frac{f(u)}{u}\,du\bigg)^p\,dt \\
& = \frac{p(p-1)}{p+1} \frac{\Gamma(p-1)\Gamma(2)}{\Gamma(p+1)} \int_0^\infty \bigg(\int_t^\infty \frac{f(u)}{u}\,du\bigg)^p\,dt = \frac{1}{p+1} \|\mathcal{C}_1^*f \|_p^p,
\end{align*}
where in the second equality we have applied Tonelli's theorem, in the third equality we used  integration by parts (with $ds = f(t)/t \,dt$), and in the fourth equality we took into account the change of variables $u=x/t$. Hence
$$\|\mathcal{C}_1^* f \|_p^p \leq  \frac{p+1}{2^p} \|\mathcal{C}_{2}^* f \|_p^p,$$ 
for all $1<p<2$ and for all $f\in \mathcal{M}^+(\mathbb{R}^+)$. 

Now we turn to the case $2\leq p <\infty$. Let $x>0$ and let $\varphi_x:(0,x] \to \mathbb{R}^+$ be the function defined by 
$$\varphi_x(v)=\bigg(\int_v^\infty f(t)(t-x) \frac{dt}{t^2}\bigg)^{p-1}.$$
By taking the Fundamental Theorem of Calculus into account, we have 
\begin{align*}
\bigg(\int_x^\infty f(t)(t-x) \frac{dt}{t^2}\bigg)^{p-1} &= -\big(\varphi_x(\infty)-\varphi_x(x)\big)= -\int_x^\infty \varphi_x'(v)\,dv \\
&=(p-1)\int_x^\infty \bigg(\int_v^\infty f(t)(t-x) \frac{dt}{t^2}\bigg)^{p-2} f(v)(v-x)\,  \frac{dv}{v^2} \\
&=(p-1)\int_x^\infty \bigg(\int_v^\infty \frac{f(t)}{t}\bigg(1-\frac{x}{t}\bigg)\, dt \bigg)^{p-2} \frac{f(v)}{v}\bigg(1-\frac{x}{v}\bigg)\, dv \\
& \geq (p-1)\int_x^\infty \bigg(\int_v^\infty \frac{f(t)}{t}\, dt \bigg)^{p-2} \frac{f(v)}{v}\bigg(1-\frac{x}{v}\bigg)^{p-1}\, dv,
\end{align*}
for all $x>0$. Hence
\begin{align*}
2^{-p} \|\mathcal{C}_2^*f \|_p^p & = \frac{p}{p+1} \int_0^\infty \bigg(\int_x^\infty \frac{f(u)}{u} \,du\bigg) \bigg(\int_x^\infty \frac{f(t)}{t} \bigg(1-\frac{x}{t}\bigg)\, dt \bigg)^{p-1}\,dx\\
& \geq \frac{p(p-1)}{p+1}  \int_0^\infty \bigg(\int_x^\infty \frac{f(u)}{u} \,du\bigg) \int_x^\infty \bigg(\int_v^\infty \frac{f(t)}{t}\, dt \bigg)^{p-2} \frac{f(v)}{v}\bigg(1-\frac{x}{v}\bigg)^{p-1}\, dv \, dx \\
& = \frac{p(p-1)}{p+1}  \int_0^\infty \bigg(\int_x^\infty \frac{f(u)}{u} \,du\bigg) \int_x^\infty \bigg(\int_v^\infty \frac{f(t)}{t}\, dt \bigg)^{p-1}\bigg(1-\frac{x}{v}\bigg)^{p-2}\frac{x}{v^2}\, dv \, dx \\
& = \frac{p(p-1)}{p+1}  \int_0^\infty  \bigg(\int_v^\infty \frac{f(t)}{t}\, dt \bigg)^{p-1}  \int_0^v \bigg(\int_x^\infty \frac{f(u)}{u} \,du\bigg) \bigg(1-\frac{x}{v}\bigg)^{p-2}\frac{x}{v^2}\, dx \, dv \\
& \geq \frac{p(p-1)}{p+1}  \int_0^\infty  \bigg(\int_v^\infty \frac{f(t)}{t}\, dt \bigg)^{p}  \int_0^v \bigg(1-\frac{x}{v}\bigg)^{p-2}\frac{x}{v^2}\, dx \, dv \\
&  = \frac{p(p-1)}{p+1}  \int_0^\infty  \bigg(\int_v^\infty \frac{f(t)}{t}\, dt \bigg)^{p}  \int_0^v \bigg(1-\frac{x}{v}\bigg)^{p-2}\frac{x}{v}\, \frac{dx}{v} \, dv \\
&  = \frac{p(p-1)}{p+1}  \int_0^\infty  \bigg(\int_v^\infty \frac{f(t)}{t}\, dt \bigg)^{p}  \int_0^1 (1-u)^{p-2}u \, du \, dv \\
& = \frac{p(p-1)}{p+1}  B(p-1,2) \|\mathcal{C}_1^*f \|_p^p  =  \frac{1}{p+1}  \|\mathcal{C}_1^*f \|_p^p,
\end{align*}
where we have applied integration by parts in the second equality, Tonelli's theorem in the third equality, the change of variable $u=x/t$ in the fourth equality. Therefore
$$\|\mathcal{C}_1^* f \|_p^p \leq  \frac{p+1}{2^p} \|\mathcal{C}_{2}^* f \|_p^p,$$
for all $2 \leq p < \infty$ and for all $f \in C_c^+(\mathbb{R}^+)$.  This completes the proof of the inequality in \eqref{IneqCn*Cn+1*}. The proof that the constant in \eqref{IneqCnCn+1} is optimal follows directly from \eqref{Limit3}.
\hfill $\square$
\vspace{\topsep}

\section{Further Comments}
\label{FurtherComments}
The table below summarizes our results for the norms of the inclusion operators
\begin{gather*}
\Lambda_n^p(\mathbb{R}^+) \hookrightarrow \Lambda_{n+k}^p(\mathbb{R}^+) \quad (1 < p \le \infty), \\
\Lambda_n^{*p}(\mathbb{R}^+) \hookrightarrow \Lambda_{n+k}^{*p}(\mathbb{R}^+) \quad (1 \le p < \infty).
\end{gather*}

\begin{table}[h]
\centering
\renewcommand{\arraystretch}{2}
\begin{tabularx}{\textwidth}{|l|*2{>{\centering\arraybackslash}X|}}
\cline{2-3}
\multicolumn{1}{c|}{} & $\begin{array}{c} \displaystyle \|i\|_{\Lambda_n^p(\mathbb{R}^+) \to \Lambda_{n+k}^p(\mathbb{R}^+)} \\[-10pt] (1 < p \le \infty) \end{array}$ & $\begin{array}{c} \displaystyle \|i\|_{\Lambda_n^{*p}(\mathbb{R}^+) \to \Lambda_{n+k}^{*p}(\mathbb{R}^+)} \\[-10pt] (1 \le p < \infty) \end{array}$ \\ \hline
$(n,k) \in \mathbb{N} \times \mathbb{N} $ & $\begin{array}{c} \displaystyle \frac{B\left(n + \frac{1}{p'}, k\right)}{B(n+1, k)} \end{array} \vphantom{\begin{array}{c} X \\ (1 < p \le \infty) \end{array}}$ & $\begin{array}{c} \displaystyle \frac{B\left(n + \frac{1}{p}, k\right)}{B(n+1, k)} \end{array} \vphantom{\begin{array}{c} X \\ (1 \le p < \infty) \end{array}}$ \\ \hline
\end{tabularx}
\end{table}

Let $n,k \in \mathbb{N}$. We have seen in Theorem~\ref{Theorem3.2.} that the norm of the inclusion operator
\[
\|i\|_{\Lambda_{n+k}^\infty(\mathbb{R}^+) \to \Lambda_n^\infty(\mathbb{R}^+)} = \frac{(n-1)^{n-1}}{n^{n-1}} \frac{(n+k)^{n+k-1}}{(n+k-1)^{n+k-1}}.
\]

The following table summarizes the norms of the inclusion operators:
\begin{gather*}
\Lambda_{n+k}^p(\mathbb{R}^+) \hookrightarrow \Lambda_n^p(\mathbb{R}^+) \quad \text{for all } 1 < p < \infty, \\
\Lambda_{n+k}^{*p}(\mathbb{R}^+) \hookrightarrow \Lambda_n^{*p}(\mathbb{R}^+) \quad \text{for all } 1 \le p < \infty,
\end{gather*}
for all $n, k \in \mathbb{N}$.

\medskip

\begin{table}[h]
\centering
\renewcommand{\arraystretch}{2}
\begin{tabularx}{\textwidth}{|c|*6{>{\centering\arraybackslash}X|}}
\cline{2-7}
\multicolumn{1}{c|}{} & \multicolumn{3}{c|}{$\begin{array}{c} \displaystyle \|i\|_{\Lambda_{n+k}^p(\mathbb{R}^+) \to \Lambda_{n}^p(\mathbb{R}^+)} \\[-10pt] (1 < p < \infty) \end{array}$} & \multicolumn{3}{c|}{$\begin{array}{c} \displaystyle \|i\|_{\Lambda_{n+k}^{*p}(\mathbb{R}^+) \to \Lambda_{n}^{*p}(\mathbb{R}^+)} \\[-10pt] (1 \le p < \infty) \end{array}$} \\ \cline{2-7}
\multicolumn{1}{c|}{} & \multicolumn{1}{c|}{$1<p<2$} & \multicolumn{1}{c|}{$p=2$} & \multicolumn{1}{c|}{$2<p<\infty$} & \multicolumn{1}{c|}{$1\leq p<2$} & \multicolumn{1}{c|}{$p=2$} & \multicolumn{1}{c|}{$2<p<\infty$} \\ \hline
$(n,k)=(1,1)$ & \multicolumn{3}{c|}{$\begin{array}{c} \displaystyle{C(n,k,p)} \end{array} \vphantom{\begin{array}{c} X \\ (1 < p \le \infty) \end{array}}$} & \multicolumn{3}{c|}{$\begin{array}{c} \displaystyle C^*(n,k,p) \end{array} \vphantom{\begin{array}{c} X \\ (1 \le p < \infty) \end{array}}$} \\ \hline
\makecell{$(n,k) \in \mathbb{N} \times \mathbb{N}$ \\ $(n,k) \ne (1,1)$} & \multicolumn{1}{c|}{$\begin{array}{c} \vphantom{\displaystyle \bigg(\frac{n}{n + k}\bigg)^p \ \frac{p(n + k-1)+1}{p(n-1)+1}} \end{array}$ \vphantom{$\begin{array}{c} X \\ (1 < p \leq \infty) \end{array}$}\hspace{-20pt} ? } & $\displaystyle C(n,k,p)$ & ? & \multicolumn{1}{c|}{$\begin{array}{c} \vphantom{\displaystyle \bigg(\frac{n}{n + k}\bigg)^p \ \frac{p(n + k-1)+1}{p(n-1)+1}} \end{array}$ \hspace{-20pt} ? } & $\displaystyle C^*(n,k,p)$ & ? \\ \hline
\end{tabularx}
\end{table}

\medskip
Here, the symbol ``?'' indicates that the exact value of the corresponding norm remains unknown, and $C(n,k,p)$ and $C^*(n,k,p)$ are defined for $n, k \in \mathbb{N}$ and $1 < p < \infty$ by
\[
C(n,k,p) = \left(\frac{n}{n+k}\right)^p \frac{B\big((n - 1)p + 1, p - 1\big)}{B\big((n + k - 1)p + 1, p - 1\big)}
\]
and
\[
C^*(n,k,p) = \left(\frac{n}{n + k}\right)^p \frac{p(n + k - 1) + 1}{p(n - 1) + 1}.
\]

\medskip

By adapting the arguments from Theorems~\ref{Theorem3.3.} and~\ref{Theorem3.4.}, together with the semigroup property of the Riemann--Liouville fractional integral operators (see \cite[Theorem~2.2]{Kai}), it can be shown that
\begin{equation}\label{question1}
0 < \|i\|_{\Lambda_{n+k}^p(\mathbb{R}^+) \to \Lambda_n^p(\mathbb{R}^+)} < \infty
\end{equation}
and
\begin{equation}\label{question2}
0 < \|i\|_{\Lambda_{n+k}^{*p}(\mathbb{R}^+) \to \Lambda_n^{*p}(\mathbb{R}^+)} < \infty
\end{equation}
for all $1 < p < \infty$ and $n, k \in \mathbb{N}$, with $p \neq 2$ and  $(n, k) \ne (1, 1)$. Unfortunately, this approach does not yield the exact values of these inclusion norms.

\medskip

In any case, since the operator $\mathcal{C}_m$ and its adjoint $\mathcal{C}_m^*$ are defined for all real number $m > 0$ (and, for instance, the inequalities \eqref{BoundednessC} and \eqref{BoundednessC*} both hold for any $m > 0$), and in view of the tables above, it is natural to state the following general conjecture, which generalizes the main results of this paper:

\bigskip

\noindent\textbf{Conjecture.} Let $1 < p < \infty$ and $n, k > 0$ be real numbers. Then:
\begin{align*}
    \|i\|_{\Lambda_n^p(\mathbb{R}^+) \to \Lambda_{n+k}^p(\mathbb{R}^+)} & =\frac{B\big(n + \frac{1}{p'}, k\big)}{B(n+1, k)}, \\
    \|i\|_{\Lambda_n^{*p}(\mathbb{R}^+) \to \Lambda_{n+k}^{*p}(\mathbb{R}^+)}&=\frac{B\big(n + \frac{1}{p}, k\big)}{B(n+1, k)}, \\
\|i\|_{\Lambda_{n+k}^p(\mathbb{R}^+) \to \Lambda_n^p(\mathbb{R}^+)}^p &= \left( \frac{n}{n+k} \right)^p \frac{B\big(p(n-1)+1, 1\big)}{B\big(p(n+k-1)+1, 1\big)}, \\[1ex]
\|i\|_{\Lambda_{n+k}^{*p}(\mathbb{R}^+) \to \Lambda_n^{*p}(\mathbb{R}^+)}^p & = \bigg(\frac{n}{n + k}\bigg)^p \ \frac{p(n + k-1)+1}{p(n-1)+1}.
\end{align*}

\medskip

\noindent\textbf{Acknowledgements:} The first author extends his sincere gratitude to his PhD supervisors, Santiago Boza and Javier Soria, for their continuous guidance and support throughout his doctoral studies. 

\medskip

\noindent\textbf{Funding:} The first author was partially supported by grants PID2024-155917NB-I00, funded by MCIN/AEI/10.13039/501100011033. The second named author has been partially supported by Project PID2022-137294NB-I00 and PID2025-169474NB-C21 of Ministry of Science and Innovation of Spain.

\medskip

\noindent\textbf{Data availability:} No data were used for the research described in the article.

\medskip

\noindent\textbf{Conflicts of Interest:} The authors declare that they have no conflict of interest.

\end{document}